\documentclass[a4paper,12pt,oneside,headsepline,draft]{scrartcl}
\pdfoutput=1

\usepackage{ifdraft}

\usepackage[utf8]{inputenc}
\usepackage[T1]{fontenc}

\usepackage{lmodern}
\usepackage{fourier}
\usepackage{amssymb, amsfonts}

\usepackage{mathrsfs} % script characters (\mathscr)
\usepackage{dsfont}
\usepackage[usenames,dvipsnames]{xcolor}%, colortbl}
\usepackage{lettrine}
\usepackage{url}

\usepackage{stmaryrd}
\usepackage[english]{babel}

\usepackage[draft=false]{scrlayer-scrpage}
\usepackage{wrapfig}
\usepackage{footmisc}
\usepackage{needspace}

\usepackage{amsmath}
\usepackage[thmmarks, amsmath, amsthm]{ntheorem}

\usepackage[style=alphabetic,giveninits,url=false,sortcites=true,maxbibnames=99]{biblatex}
\bibliography{bibliography.bib}

\ifdraft{{}}
  {\usepackage[pdftex,
               pdfborder={0 0 0}, colorlinks=true,
               linkcolor=BrickRed, citecolor=ForestGreen, urlcolor=RoyalBlue]{hyperref}}

\ifdraft{\usepackage[colorinlistoftodos]{todonotes}}{}

\usepackage{lineno}
\usepackage{booktabs}
\usepackage[inline]{enumitem}
\usepackage{float}
\usepackage{graphicx}
\usepackage{multicol}
\ifdraft{\date{\today}}{\date{}}

\KOMAoptions{DIV=13}

\clearscrheadings
\automark[section]{subsection}

\renewcommand{\subsectionmark}[1]{}

\lohead{{\small \headertitle}}
\rohead{{\small \headerauthors}}

\RedeclareSectionCommand[%
  font=\Large\sffamily\bfseries,%
  beforeskip=1\baselineskip,%
  afterskip=0.5\baselineskip,%
  indent=0em%
  ]{section}

\RedeclareSectionCommands[%
  font=\normalfont\bfseries,%
  afterskip=-1em%
  ]{subsection,subsubsection}

\RedeclareSectionCommands[%
  font=\normalfont\itshape,%
  afterskip=-1em,%
  indent=0pt,
  ]{paragraph}

\AtEveryBibitem{\clearfield{isbn}}
\AtEveryBibitem{\clearlist{language}}
\AtEveryBibitem{\clearfield{pages}}

\newenvironment{enumeratearabic}{
\begin{enumerate}[label=(\arabic*), leftmargin=0pt,labelindent=2em,itemindent=!]
}{
\end{enumerate}
}

\newenvironment{enumeratearabic*}{
\begin{enumerate*}[label=(\arabic*)] %%, leftmargin=0pt,labelindent=2em,itemindent=!]
}{
\end{enumerate*}
}

\newenvironment{enumerateroman*}{
\begin{enumerate*}[label=(\roman*)] %%, leftmargin=0pt,labelindent=2em,itemindent=!]
}{
\end{enumerate*}
}

\numberwithin{equation}{section}

\newtheorem{theoremcounter}{theoremcounter}[section]

\theoremstyle{plain}

\newtheorem{corollary}[theoremcounter]{Corollary}
\newtheorem{lemma}[theoremcounter]{Lemma}
\newtheorem{proposition}[theoremcounter]{Proposition}
\newtheorem{theorem}[theoremcounter]{Theorem}

\theoremnumbering{Roman}

\theoremstyle{definition}

\theoremstyle{remark}

\newtheorem{remark}[theoremcounter]{Remark}

\newtheorem*{mainremark}{Remark}

\newtheorem*{remarkcomputation}{Computation}

\let\cal\undefined
\ifdraft{
 
}{
 
}

\newcommand{\tx}{\ensuremath{\text}}

\newcommand{\bboard}{\ensuremath{\mathbb}}
\newcommand{\cal}{\ensuremath{\mathcal}}
\renewcommand{\frak}{\ensuremath{\mathfrak}}

\newcommand{\ga}{\ensuremath{\gamma}}

\newcommand{\bA}{\ensuremath{\bboard A}}

\newcommand{\bD}{\ensuremath{\bboard D}}

\newcommand{\bH}{\ensuremath{\bboard H}}

\newcommand{\bP}{\ensuremath{\bboard P}}

\newcommand{\bR}{\ensuremath{\bboard R}}

\newcommand{\bZ}{\ensuremath{\bboard Z}}

\newcommand{\cF}{\ensuremath{\cal{F}}}

\newcommand{\cH}{\ensuremath{\cal{H}}}

\newcommand{\cL}{\ensuremath{\cal{L}}}

\newcommand{\cO}{\ensuremath{\cal{O}}}

\newcommand{\cY}{\ensuremath{\cal{Y}}}

\newcommand{\frakm}{\ensuremath{\frak{m}}}

\newcommand{\rmd}{\ensuremath{\mathrm{d}}}

\newcommand{\rmJ}{\ensuremath{\mathrm{J}}}

\newcommand{\rmM}{\ensuremath{\mathrm{M}}}

\newcommand{\rmU}{\ensuremath{\mathrm{U}}}

\newcommand*{\longhookrightarrow}{\ensuremath{\lhook\joinrel\relbar\joinrel\rightarrow}}

\newcommand{\lra}{\ensuremath{\longrightarrow}}
\newcommand{\lhra}{\ensuremath{\longhookrightarrow}}

\newcommand{\lmto}{\ensuremath{\longmapsto}}

\newcommand{\N}{\ensuremath{\mathbb{N}}}
\newcommand{\Z}{\ensuremath{\mathbb{Z}}}
\newcommand{\Q}{\ensuremath{\mathbb{Q}}}
\newcommand{\R}{\ensuremath{\mathbb{R}}}
\newcommand{\C}{\ensuremath{\mathbb{C}}}

\newcommand{\OK}{\ensuremath{\mathcal{O}_E}}
\renewcommand{\pmod}[1]{\ensuremath{\;(\mathrm{mod}\, #1)}}

\newcommand{\Hom}{\ensuremath{\mathop{\mathrm{Hom}}}}

\newenvironment{psmatrix}{\left(\begin{smallmatrix}}{\end{smallmatrix}\right)}

\newcommand{\Mat}[1]{\ensuremath{\mathrm{Mat}_{#1}}}
\newcommand{\GL}[1]{\ensuremath{\mathrm{GL}_{#1}}}

\newcommand{\SL}[1]{\ensuremath{\mathrm{SL}_{#1}}}

\newcommand{\Sp}[1]{\ensuremath{\mathrm{Sp}_{#1}}}

\newcommand{\ord}{\ensuremath{\mathrm{ord}}}

\newcommand{\Tr}{\ensuremath{\mathop{\mathrm{Tr}}}}

\renewcommand{\det}{\ensuremath{\mathrm{det}}}

\renewcommand{\ker}{\ensuremath{\mathop{\mathrm{ker}}}}

\newcommand{\UGZ}{\ensuremath{\mathrm{U} (g, g) (\mathbb{Z})}}

\newcommand{\UGQ}{\ensuremath{\mathrm{U} (g, g) (\mathbb{Q})}}

\newcommand{\UGR}{\ensuremath{\mathrm{U} (g, g) (\mathbb{R})}}

\newcommand{\MG}{\ensuremath{\mathrm{M}_\bullet^{(g)}}}
\newcommand{\MGA}{\ensuremath{\mathrm{M}_k(\Gamma)}}
\newcommand{\MGR}{\ensuremath{\mathrm{M}_\bullet^{(g)}(\rho)}}
\newcommand{\MGP}{\ensuremath{\mathrm{M}_\bullet^{(g^\prime)}(\rho)}}
\newcommand{\MA}{\ensuremath{\mathrm{M}_k^{(g)}}}
\newcommand{\MK}{\ensuremath{\mathrm{M}_{\leq k}^{(g)}}}
\newcommand{\MR}{\ensuremath{\mathrm{M}_k^{(g)}(\rho)}}

\newcommand{\FMA}{\ensuremath{\mathrm{FM}_k^{(g)}}}
\newcommand{\FMAA}{\ensuremath{\mathrm{FM}_k^{(g)}(\rho)}}
\newcommand{\FMN}{\ensuremath{\mathrm{FM}_k^{(n)}(\rho)}}
\newcommand{\FMK}{\ensuremath{\mathrm{FM}_{\leq k}^{(g)}}}
\newcommand{\FMO}{\ensuremath{\mathrm{FM}_{k_0}^{(g)}}}

\newcommand{\FMC}{\ensuremath{\mathrm{FM}_k^{(g, l)}}}
\newcommand{\FMR}{\ensuremath{\mathrm{FM}_k^{(g, l)} (\rho)}}
\newcommand{\FMG}{\ensuremath{\mathrm{FM}_\bullet^{(g)}}}
\newcommand{\FMGL}{\ensuremath{\mathrm{FM}_\bullet^{(g, l)}}}
\newcommand{\FMGR}{\ensuremath{\mathrm{FM}_\bullet^{(g, l)} (\rho)}}
\newcommand{\FMGP}{\ensuremath{\mathrm{FM}_\bullet^{(g^\prime)} (\rho)}}

\newcommand{\JA}{\ensuremath{\mathrm{J}_{k, m}^{(g)}}}

\newcommand{\JB}{\ensuremath{\mathrm{J}_{k, m}^{(g - 1)}}}

\newcommand{\JR}{\ensuremath{\mathrm{J}_{k, m}^{(g)} (\rho)}}
\newcommand{\JAA}{\ensuremath{\mathrm{J}_{k, m}^{(g - 1)} (\rho)}}
\newcommand{\Mgl}{\ensuremath{\mathrm{Mat}_{g,l} (E)}}
\newcommand{\MglA}{\ensuremath{\mathrm{Mat}_{g,l}}}
\newcommand{\MglZ}{\ensuremath{\mathrm{Mat}_{g,l} (\cO_E)}}
\newcommand{\MglC}{\ensuremath{\mathrm{Mat}_{g,l} (\mathbb{C})}}

\newcommand{\MlgZ}{\ensuremath{\mathrm{Mat}_{l, g} (\cO_E)}}
\newcommand{\MlgC}{\ensuremath{\mathrm{Mat}_{l, g} (\mathbb{C})}}
\newcommand{\Hg}{\ensuremath{\mathbb{H}_g}}
\newcommand{\Hgl}{\ensuremath{\mathbb{H}_{g, l}}}
\newcommand{\HGK}{\ensuremath{\mathrm{Herm}_{g} (E)}}

\newcommand{\HLK}{\ensuremath{\mathrm{Herm}_{l} (E)}}

\newcommand{\Ggl}{\ensuremath{\Gamma^{(g, l)}}}

\newcommand{\rot}{\ensuremath{\mathrm{rot}}}

\title{Some cases of Kudla's modularity conjecture for unitary Shimura varieties}
\author{Jiacheng Xia}
\date{}

\newcommand{\headertitle}{{\normalfont Some cases of Kudla's modularity conjecture for unitary Shimura varieties
}}
\newcommand{\headerauthors}{
  Jiacheng~Xia
}
\begin{document}
\maketitle
\begin{abstract}
\textbf{Abstract:}   We use the method of Bruinier--Raum to show that symmetric formal Fourier--Jacobi series, in the cases of norm-Euclidean imaginary quadratic fields, are Hermitian modular forms. Consequently, combining a theorem of Yifeng Liu, we deduce Kudla's conjecture on the modularity of generating series of special cycles of arbitrary codimension for unitary Shimura varieties defined in these cases. \end{abstract}

%\tableofcontents

%\begin{plainfootnotes}
%\begin{flushleft}
%{\fontfamily{lms}\sffamily
%  \hspace{18pt}{\huge Some cases of Kudla's Modularity Conjecture for unitary Shimura Varieties}
%}
% \\{\fontfamily{lms}\sffamily
%   \hspace{20pt}%
%   %% SUBTITLE
% }
%\\[.6em]\hspace{20pt}{\large%
%  Jiacheng Xia%
%  \footnote{The author was partially supported by Vetenskapsr\aa det Grant~2015-04139.}
%}
%\\[1.2em]
%\end{flushleft}
%\end{plainfootnotes}

\thispagestyle{scrplain}

%%%%%%%%%%%%%%%%%%%%%%%%%%%%%%%%%%%%%%%%%%%%%%%%%%
%%% INTRODUCTION
%%% Intended to make this part an abstract but don't really know the format.
%\Needspace*{4em}
%\addcontentsline{toc}{section}{Introduction}
%\markright{Introduction}
%\lettrine[lines=2,nindent=.2em]{\tbf W}{e}  prove 

%%%%%%%%%%%%%%%%%%%%%%%%%%%%%%%%%%%%%%%%%%%%%%%%%%
%%% PAPER BODY

\section{Introduction}\label{sec:intro}
Fourier--Jacobi expansions of automorphic forms, first defined in \cite{piatetsky-shapiro-1966}, are among the major tools to study the subject. For Siegel modular forms, they played prominent roles in the proof of Saito--Kurokawa conjecture \cite{andrianov-1979, maass-1979a, maass-1979b, maass-1979c, zagier-1981}, and recently in the work of Bruinier--Raum \cite{bruinier-raum-2015} on Kudla's modularity conjecture for orthogonal Shimura varieties. Fourier--Jacobi expansions are also available for unitary groups. For example, for holomorphic automorphic forms on unitary groups $\mathrm{U}(2, 1)$ over totally real number fields, Takuro Shintani developed a theory (1979) on Fourier--Jacobi expansions which was reformulated in \cite{murase-sugano-2002}; as an application of Shintani's theory, a nonvanishing criterion for the unitary Kudla lift (\cite{kudla-1981}) of a holomorphic cusp form on $\mathrm{U} (1, 1)$ was found by Murase--Sugano \cite{murase-sugano-2007}. In this manuscript, we define a formal analogue of Fourier--Jacobi expansions for the unitary group $\mathrm{U}(g, g)$, which is called symmetric formal Fourier--Jacobi series after Bruinier--Raum, and show that they define Hermitian modular forms \cite{braun-1949, braun-1950, braun-1951} in the cases of norm-Euclidean imaginary quadratic fields.

Our major motivation to study such symmetric formal Fourier--Jacobi series arises from the role they play for the Kudla conjecture \cite{kudla-1997} on the modularity of generating series of special cycles, for both orthogonal and unitary Shimura varieties. In the unitary case, Kudla's conjecture predicts that the special cycles on a unitary Shimura variety should be the Fourier coefficients of some Hermitian modular form. Beyond the case of codimension $1$, assuming absolute convergence of the generating series, Yifeng Liu \cite{MR2928563} proved the unitary Kudla conjecture. Recently, Yota Maeda \cite{Maeda-2021} gave another proof of Liu's result and generalized it assuming the Bloch--Beilinson conjecture. In this manuscript, we show that the  unitary Kudla conjecture is true unconditionally in the cases of norm-Euclidean imaginary quadratic fields. And the way we do that is combining a fact from Liu's theorem that the generating series is a symmetric formal Fourier--Jacobi series and our modularity result that such a series defines a Hermitian modular form. This is a method that has already been successfully applied to prove the orthogonal Kudla conjecture over $\Q$ in the work of Bruinier--Raum \cite{bruinier-raum-2015}. In the unitary case of imaginary quadratic fields, the scope of this method is limited to the cases of norm-Euclidean imaginary quadratic fields that we treat in this manuscript.

\subsection{The modularity result}
Let $E/\Q$ be an imaginary quadratic field. For integers $g, k, l$ such that $1 \leq l \leq g - 1$, every Hermitian modular form $f$ of degree $g$, weight $k$ has a Fourier--Jacobi expansion of cogenus $l$.  More precisely, if we write the variable $\tau \in \bH_g$ in the Hermitian upper half space $\Hg$ as
\[\tau = \begin{pmatrix}
\tau_1 & w\\
z & \tau_2
\end{pmatrix} \]
for $\tau_1 \in \bH_{g - l}$, $\tau_2 \in \bH_l$, $w \in  \Mat{g - l, l} (\C)$, and $z \in \Mat{l, g - l} (\C)$, then  $f$ has a Fourier--Jacobi expansion of the form
\begin{gather*}\label{apricot}
   f (\tau) = 
\sum_{m \in \mathrm{Herm}_l (E)_{\geq 0}} \phi_m (\tau_1, w, z) e (m \tau_2)
\tx{,} 
\end{gather*}
where the sum runs over all the $l \times l$ positive semidefinite Hermitian matrices $m$ with entries in $E$, and $e (x) = \exp \big(2 \pi i \cdot \Tr(x) \big)$ for a square matrix $x$. Moreover, the coefficients $\phi_m$ are Hermitian Jacobi forms of degree $g - l$, weight $k$, and index $m$, which satisfy certain symmetry condition for their ordinary Fourier coefficients from the modularity of $f$.  

The fact that any Hermitian modular form is a formal Fourier--Jacobi series satisfying this symmetry condition and absolute convergence, motivates  the notion of symmetric formal Fourier--Jacobi series, namely such series that satisfy the symmetry condition without assuming absolute convergence. Conversely, if a symmetric formal Fourier--Jacobi series converges absolutely, then it defines a Hermitian modular form. One natural question is then: does every symmetric formal Fourier--Jacobi series automatically converge absolutely?

Such a rigidity result was already considered by Ibukiyama--Poor--Yuen \cite{ibukiyama-poor-yuen-2012} for Siegel paramodular forms. For Siegel modular forms, J. Bruinier \cite{bruinier-2015} and M. Raum \cite{raum-2015c}  resolved independently the case of degree $2$ and arbitrary type over $\Q$. In their joint work \cite{bruinier-raum-2015}, Bruinier--Raum proved the general case of higher degree and arbitrary type over $\Q$. In this manuscript, we show the Hermitian counterpart in the cases of norm-Euclidean imaginary quadratic fields $E = \Q(\sqrt {d})$, namely for $d \in \big\{-1, -2, -3, -7, -11\big\}$. 

\begin{theorem}\label{castle}
Every symmetric formal Fourier--Jacobi series of arbitrary arithmetic type for the unitary group $\UGZ$, defined in the cases of norm-Euclidean imaginary quadratic fields $E$, converges absolutely to  a Hermitian modular form. 
\end{theorem}

\begin{remark}
A more detailed version is presented in Theorem~\ref{heaven}.
\end{remark}

\begin{remark}
After uploading a first version of this manuscript on arXiv, the author learned that Yuxiang Wang partially proved similar results in his thesis work \cite{Wang-2020}. Note that Wang claims to prove for all imaginary quadratic fields. However, in the proof of Lemma~4.6 on page 39 of his thesis, the choice of $r$ is not justified, and in fact cannot be made in the general case. One of the first counterexamples arises from the case of $E = \Q (\sqrt{-5})$, where in Wang's notations for an arbitrarily fixed $a \in \frac{1}{m} (\cO^\#)^g/\cO^g$, one cannot even choose $r = (r_1, \ldots, r_g)$ such that  $\frac{1}{m} r \equiv a \pmod{\cO^g}$ and that $|r_g|^2 < m^2$ (but Wang claims $|r_g|^2 \leq \frac{m^2}{D}$ for $D = 20$).
\end{remark}

\subsection{Application: The unitary Kudla conjecture}
The modularity of generating series of geometric cycles has been studied since the first construction of such modular generalting series in the work of Hirzebruch--Zagier \cite{hirzebruch-zagier-1976}. In a long collaboration, Kudla--Millson \cite{kudla-millson-1986, kudla-millson-1987, kudla-millson-1990} defined a special family of locally symmetric cycles of Riemmanian locally symmetric spaces $X$, which are called special cycles, and proved the modularity of their generating series valued in cohomology classes. It turns out to be particularly interesting when $X$ is a Shimura variety, as the analogous generating series valued in Chow groups can be defined, and it is natural to ask if they are already modular at this level. 

In the case of Shimura varieties of orthogonal type over a totally real number field, Stephen S. Kudla raised this question in his seminal work \cite{kudla-1997}. Inspired by the work of Gross--Kohnen--Zagier \cite{gross-kohnen-zagier-1987} on the images of Heegner points in the Jacobian of a modular curve, Richard Borcherds \cite{borcherds-1999} proved the modularity of generating series of Heegner (special) divisors valued in the first Chow group, by employing his work \cite{borcherds-1998} on the construction of a family of meromorphic modular functions via regularized theta lift. Building upon Borcherds' work, W. Zhang \cite{zhang-2009} proved that the generating series of special cycles valued in Chow groups are modular assuming absolute convergence. Subsequently, Bruinier--Raum \cite{bruinier-raum-2015} completed the proof of Kudla's modularity conjecture over $\Q$. Most recently, over an arbitrary totally real field of degree $d$ assuming the Bloch--Beilinson conjecture on the injectivity of the Abel--Jacobi maps, Kudla \cite{Kudla-2019} proved the modularity conjecture for orthogonal Shimura varieties of signature $((m, 2)^{d_+}, (m + 2, 0)^{d - d_+})$, and Yota Maeda \cite{maeda-2020a} independently showed the modularity conjecture in more general cases assuming absolutely convergence of the generating functions and the Bloch--Beilinson conjecture.

In the case of Shimura varieties of unitary type and codimension $1$ (special divisors), the conjecture was verified in \cite{MR2928563}. In recent preprints of Bruinier--Howard--Kudla--Rapoport--Yang (to appear in Ast\'erisque) \cite{Bruinier--Howard--Kudla--Rapoport--Yang-1, Bruinier--Howard--Kudla--Rapoport--Yang-2}, generating series of special divisors valued in the Chow group and the arithmetic Chow group are defined on the compactified integral model; consequently, their modularity is proven and more arithmetic applications are found, including relations between derivatives of $L$-functions and arithmetic intersection pairings \`a  la Gross--Zagier, and a special case of Colmez's conjecture on the Faltings heights of abelian varieties with complex multiplication. However, the cases of higher codimension remain open problems. In the cases of norm-Euclidean imaginary quadratic fields, we show in this manuscript that the convergence assumption in Theorem 3.5 of Y. Liu's work \cite{MR2928563} can be dropped for arbitrary codimension. The following theorem is shown in Section~\ref{sit}.
\begin{theorem}
In the cases of norm-Euclidean imaginary quadratic fields, the unitary Kudla conjecture is true for open Shimura varieties.
\end{theorem}

Since modular forms of a fixed weight are finite-dimensional, an immediate consequence of Kudla's modularity conjecture is that the $\C$-ranks of the special cycles in the complexification of the Chow groups are explicitly bounded from above, even though we don't know in general whether the Chow groups are finite-dimensional. Furthermore, as an appealing consequence, relations between Fourier coefficients of certain Hermitian modular forms give rise to the corresponding relations between special cycles, which are otherwise not accessible in the literature.

\subsection{Structure of the proof}

The proof of our main result can be separated into three parts.

In the first part (Section~\ref{run}), we show that every symmetric formal Fourier–Jacobi series $f$ of genus $g$, cogenus $1$, and trivial type is algebraic over the graded algebra of Hermitian modular forms of genus $g$. This part is where the norm-Euclidean condition is required.

In the second part (Section~\ref{fly}), we prove in two steps that every such $f$ converges on the whole Hermitian upper half space. First, the algebraicity of $f$ allows us to show the local convergence of $f$ in a neighborhood of any toroidal boundary of the Hermitian modular variety. Then we show that $f$ can be analytically continued to the whole space. We develop a technique of embedding the Siegel upper half space to the Hermitian one in a generic yet rational way to pass to the Siegel case, which was resolved by Bruinier--Raum. 

Finally in the third part (Section~\ref{jump}), to cover all arithmetic types and arbitrary cogenus, we prove in two steps by induction. First, fixing the values of genus and cogenus, we consider a natural pairing of (meromorphic) Hermitian modular forms and (meromorphic) formal symmetric Fourier--Jacobi series. Then, we increase the cogenus from $1$ by induction, using the machinery introduced in Section~\ref{walk}.

\section{Hermitian modular forms and symmetric formal Fourier--Jacobi series}\label{walk}
Our aim in this section is to provide preliminaries for this manuscript and introduce the notion of symmetric formal Fourier--Jacobi series in the Hermitian case. First we recall briefly Hermitian modular forms in Subsections~\ref{harmony} and~\ref{peace}, and introduce notations which are used throughout the manuscript. Then we define Hermitian Jacobi forms of higher degrees and recall a key lemma connecting the Hermitian modular group to certain Jacobi group in Subsections~\ref{tranquility} and~\ref{serenity}. Based on the notations introduced in Subsections~\ref{harmony} and~\ref{peace}, we recall Fourier--Jacobi expansions in Subsection~\ref{placidity} and introduce its formal analogue, symmetric formal Fourier--Jacobi series, in Subsection~\ref{calmness}. Finally, we study a formal analogue of theta decomposition of Jacobi forms and prove a few propositions in Subsections~\ref{Alice} and~\ref{whispering}, which are used in Section~\ref{run} and in the proof of Proposition~\ref{holmes}.

\subsection{Hermitian modular groups}\label{harmony}
Historically in the context of Hermitian modular forms, the unitary group $\mathrm{U}(g, g)$ is defined as a classical group, for instance in \cite{braun-1949}.  In most parts of Sections~\ref{walk}-\ref{jump}, following all prior literature in this field, we employ this classical notion as opposed to the perspective of algebraic groups, but adopt the notation from the latter for coherence of this manuscript. Let $E/\Q$ be an imaginary quadratic field, specified at the beginning of each section, and let $\mathrm{U}(g, g) (\R)$ denote the classical unitary group of signature $(g, g)$ with entries in $\C$, which in fact corresponds to the real points of a reductive group. Similarly, let $\mathrm{U}(g, g) (\Q)$ denote the classical unitary group with entries in $E$, and $\mathrm{U}(g, g) (\Z)$ the one with entries in the ring of integers $\OK \subseteq \C$ for a fixed embedding $\iota: E \lhra \C$. In other words, the classical group $\mathrm{U}(g, g) (\R)$ (resp. $\mathrm{U}(g, g) (\Q)$ and $\mathrm{U}(g, g) (\Z)$) is the subgroup of $\GL{2 g} (\C)$ (resp. $\GL{2 g} (E)$ and $\GL{2 g} (\OK)$) whose elements $\ga$ satisfy
the equation
\begin{linenomath}
\begin{gather}\label{banana}
\ga^\ast J_{g, g} \ga = J_{g, g}\,
\tx{,}
\end{gather}
\end{linenomath}
for the $2g \times 2g$ matrix
\begin{linenomath}
\begin{gather*}
J_{g, g} =
\begin{psmatrix}
0_g & I_g \\
- I_g & 0_g
\end{psmatrix}\tx{.}
\end{gather*}
\end{linenomath} 
Writing $\ga$ furthermore in blocks of $g \times g$ matrices
\begin{linenomath}
\begin{gather*}
\ga =
\begin{pmatrix}
a & b \\
c & d
\end{pmatrix}\tx{,}
\end{gather*}
\end{linenomath}
we can characterize these elements $\ga$ by the condition
\begin{linenomath}
\begin{gather}\label{pear}
a^\ast c = c^\ast a,\, b^\ast d = d^\ast b ,\, \tx{and}\, a^\ast d - c^\ast b = I_g\tx{.}
\end{gather}
\end{linenomath} 
Similar to the Siegel modular group $\Sp{g} (\Z) = \UGZ \cap \Mat{2g}(\Z)$, we call $\UGZ$ the Hermitian modular group of degree $g$. The notation $\mathrm{Herm}_n (A)$ denotes the set of $n \times n$ Hermitian matrices with entries in the ring $A \subseteq \C$. For a matrix $x \in \mathrm{Herm}_n (A)$, the notation $x \geq 0$ (resp. $x > 0$) means $x$ is positive semidefinite (resp. positive definite), and the set of all such matrices are denoted by $\mathrm{Herm}_n (A)_{\geq 0}$ (resp. $\mathrm{Herm}_n (A)_{> 0}$).

\subsection{Hermitian modular forms}\label{peace}

Basic notions on Hermitian modular forms can be found in \cite{braun-1949, braun-1950, braun-1951}. For completeness, we recall them briefly in the case of integral weight and arbitrary arithmetic type.

We define the Hermitian upper half-space of degree $g$, denoted by $\Hg$, to be the set of matrices $\tau \in \Mat{g} (\C)$ such that the Hermitian matrix
\begin{linenomath}
\begin{gather}
    y := \frac{1}{2i} (\tau - \tau^\ast)
\end{gather}
\end{linenomath}
is positive definite. By this definition, the Siegel upper half-space $\cH_g$ is the set of all the symmetric matrices in $\Hg$. Recall that the symplectic group $\Sp{g} (\R) = \UGR \cap \Mat{2g}(\R)$ acts on $\cH_g$  biholomorphically. Similarly, $\UGR$ acts on $\Hg$ biholomorphically via M\"obius transformations
\begin{linenomath}
\begin{gather}\label{blueberry}
    \tau \lmto \ga \tau := (a \tau + b) (c \tau + d)^{-1}
\end{gather}
\end{linenomath}
for 
$\ga =
\begin{psmatrix}
a & b \\
c & d
\end{psmatrix} \in \UGR$. We define the factor of automorphy $j$ via the formula
\begin{linenomath}
\begin{gather*}
    j (\ga, \tau):=\det (c \tau + d)
\end{gather*}
\end{linenomath}
for $\ga = \begin{psmatrix}
a & b \\
c & d
\end{psmatrix} \in \UGR$ and $\tau \in \Hg$. It is well known that $j (\ga, \tau)$ is always nonzero and satisfies the $1$-cocycle condition
\begin{linenomath}
\begin{gather*}
    j(\ga_1 \ga_2, \tau) 
=
j(\ga_1, \ga_2 \tau) j (\ga_2, \tau)
\end{gather*}
\end{linenomath}
for all $\ga_1, \ga_2 \in \UGR$.

Let $g, k$ be integers such that $g \geq 1$, and $\big(\rho, V(\rho)\big)$ a finite-dimensional complex representation of $\UGZ$ that factors through a finite quotient, also called an (arithmetic) type. We define the space of (vector-valued) Hermitian modular forms of degree $g$, weight $k$ and (arithmetic) type $\rho$, to be the space of holomorphic (vector-valued) functions $f: \Hg \lra V(\rho)$ satisfying the modularity condition
\begin{linenomath}
\begin{gather*}
   f(\ga \tau) = j(\ga, \tau)^k \rho(\ga) f(\tau) 
\end{gather*}
\end{linenomath}
for all $\ga \in \UGZ$, and if $g = 1$, being also holomorphic at the cusp of $\rmU (1, 1) (\Z)$. 

If $\rho$ is trivial, $\MR$ is just a direct sum of copies of classical (scalar valued) Hermitian modular forms, and the latter is denoted by $\MA$ in this manuscript. For a finite-index subgroup $\Gamma \subseteq \UGZ$, the space of classical modular forms for $\Gamma$ is denoted by $\MGA$. The graded algebra of classical Hermitian modular forms is denoted by 
\begin{linenomath}
\begin{gather*}
   \MG:= \bigoplus_{k \in \Z} \MA \tx{.}
\end{gather*}
\end{linenomath}
\begin{remark}
In the case of half-integral weight, there exist multiplier systems on some congruence subgroups $\Gamma \subseteq \UGZ$ by the general construction of Deligne \cite{deligne-1996, brylinski-deligne-2001} and the work of Prasad--Rapinchuk \cite{prasad-rapinchuk} and Prasad \cite{prasad-2004}. For the aim of the unitary Kudla conjecture in this manuscript, we work with Hermitian modular forms of integral weights.
\end{remark}

\subsection{Hermitian Jacobi groups}\label{tranquility}
Let $g, l$ be positive integers. Consider a discrete Heisenberg group
\begin{linenomath}
\begin{gather*}
    H_{\OK}^{(g, l)} :=
    \big\{
    [(\lambda, \mu), \kappa] :
    \lambda, \mu \in \MlgZ, 
    \kappa \in \mathrm{Mat}_{l} (\OK),
    \kappa + \mu \lambda^\ast \in \mathrm{Herm}_l (\OK)
    \big\}
    \tx{,}
\end{gather*}
\end{linenomath}
with the following group law
\begin{linenomath}
\begin{gather*}
    [(\lambda, \mu), \kappa]
    [(\lambda^\prime, \mu^\prime), \kappa^\prime]
    =
    [(\lambda + \lambda^\prime, \mu + \mu^\prime), 
    \kappa + \kappa^\prime + \lambda {\mu^\prime}^\ast - \mu {\lambda^\prime}^\ast]
    \tx{.}
\end{gather*}
\end{linenomath}
Note that the condition $\kappa + \mu \lambda^\ast \in \mathrm{Herm}_l (\OK)$ is equivalent to $\kappa - \lambda \mu^\ast \in \mathrm{Herm}_l (\OK)$. There is a natural action of the Hermitian modular group $\UGZ$ on the Heisenberg group $H_{\OK}^{(g, l)}$ via matrix multiplication from the right, and we define the Hermitian Jacobi group $\Ggl$ of genus $g$ and cogenus $l$ to be the semidirect product $\Ggl = \UGZ \ltimes H_{\OK}^{(g, l)}$ with the associated group law. More explicitly, the group law of the Jacobi group $\Ggl$ reads 
\begin{linenomath}
\begin{gather*}
    \big(
    \ga, [(\lambda, \mu), \kappa]
    \big)
    \big(
    \ga^\prime, [(\lambda^\prime, \mu^\prime), \kappa^\prime]
    \big)
    =
    \big(
    \ga \ga^\prime, [(\lambda + \lambda^\prime, \mu + \mu^\prime), 
    \kappa + \kappa^\prime + \lambda \ga^\prime {\mu^\prime}^\ast - \mu \ga^\prime {\lambda^\prime}^\ast]
    \big)
    \tx{.}
\end{gather*}
\end{linenomath}
Based on the standard embeddings of Heisenberg groups, we define an embedding
\begin{linenomath}
\begin{gather}\label{kiwi}
    \Ggl \lhra \rmU(g + l, g + l) (\Z)
\end{gather}
\end{linenomath}
of the Jacobi group into the Hermitian modular group via the formula
\begin{linenomath}
\begin{gather*}
\Bigg(
    \begin{pmatrix}
    a & b \\
    c & d
    \end{pmatrix}, [(\lambda, \mu), \kappa]
\Bigg)
    \mapsto
    \begin{pmatrix}
    a & 0 & b & 0 \\
    0 & I_l & 0 & 0 \\
    c & 0 & d & 0\\
    0 & 0 & 0 & I_l
    \end{pmatrix}
    \begin{pmatrix}
    I_g & 0 & 0 & \mu^\ast \\
    \lambda & I_l & \mu & \kappa \\
    0 & 0 & I_g & - \lambda^\ast\\
    0 & 0 & 0 & I_l
    \end{pmatrix}
    \tx{.}
\end{gather*}
\end{linenomath}

To relate the embedded image of the Jacobi group to the whole Hermitian modular group, we record the following fact. First we define an embedding $\rot : \GL{g}(\cO_E) \lra \UGZ$ of groups via $u \lmto \rot (u) := \begin{psmatrix}
u & 0 \\
0 & {u^\ast}^{-1}
\end{psmatrix}$.
\begin{lemma}\label{toy}
Let $g \geq 2$ be an integer, and let $E$ be an imaginary quadratic field with class number $1$. Then, the Hermitian modular group $\UGZ$ is generated by two subgroups, $\rot \big( \GL{g}(\cO_E) \big)$ and the embedded Jacobi group $\Gamma^{(g - 1, 1)}$ of cogenus $1$ under \eqref{kiwi}. 
\end{lemma}
\begin{proof}
For each integer $j$ such that $1 \leq j \leq g$, let $\iota^{(j)}: \rmU (1, 1) (\Z) \lra \UGZ$ be the $j$-th diagonal embedding sending an element $\ga = \begin{psmatrix}
a & b \\
c & d
\end{psmatrix}$ to $\iota^{(j)} (\ga)$, whose $2 \times 2$ submatrix indexed by $(j, j), (j, j + g), (j + g, j), (j + g, j + g)$ is equal to $\ga$, and the remaining entries of $\iota^{(j)} (\ga)$ are those of the identity matrix $I_{2 g}$. Since $\OK$ is a PID, we can follow the lines of Proposition 6 in Section I.3 of \cite{klingen-1990} to show that $\UGZ$ is generated by the embedded subgroups $\rot (\GL{g}(\cO_E))$ and $\iota^{(g)} (\rmU (1, 1) (\Z))$. We complete the proof by observing that any element $\iota^{(j)} (\ga)$ for $\ga \in \rmU (1, 1) (\Z)$ can be transformed to some element in the Jacobi group $\Gamma^{(g - 1, 1)}$ by conjugation with a suitable element of $\rot (\GL{g}(\cO_E))$.
\end{proof}

\subsection{Hermitian Jacobi forms}\label{serenity}
Jacobi forms are closely related to Fourier--Jacobi expansions of automorphic forms. In the case of degree $(1, 1)$ associated to Siegel modular forms, Eichler-Zagier developed a theory of Jacobi forms in their monograph \cite{eichler-zagier-1985}. The case of degree $(1, 1)$ associated to Hermitian modular forms over imaginary quadratic fields was treated by K. Haverkamp \cite{Haverkamp-thesis-1995, MR1377681}. Cases of higher degrees were investigated in \cite{murase-1989, shimura-1978a, jamazaki-1986}, and in the spirit of Eichler-Zagier a theory was built up by C.Ziegler \cite{ziegler-1989}.  To fit the presentation of this manuscript and to simplify the expressions, we define Hermitian Jacobi forms in terms of Hermitian modular forms via the embedding of Jacobi groups into Hermitian modular groups.

Let $g, l$ be positive integers. We define the Hermitian-Jacobi upper half-space $\Hgl$ of genus $g$ and cogenus $l$ to be $\Hgl := \Hg \times \MglC \times \MlgC$, and there is a (biholomorphic) action of the Jacobi group $\Ggl$ on $\Hgl$ via the formula
\begin{linenomath}
\begin{align*}
    &\Bigg(
    \begin{pmatrix}
    a & b \\
    c & d
    \end{pmatrix}, [(\lambda, \mu), \kappa]
    \Bigg)
    (\tau, w, z)\\
    =&
    \big(
    (a \tau + b)(c \tau + d)^{-1}, (c \tau + d)^{-1} (w + \tau^\ast \lambda^\ast + \mu^\ast), (z + \lambda \tau + \mu) (c \tau + d)^{-1}
    \big)
    \tx{.}
\end{align*}
\end{linenomath}
Note that the Hermitian upper half space $\bH_{g + l}$ projects to the Jacobi upper half space $\Hgl$ by restricting to the corresponding matrix blocks. For an imaginary quadratic field $E/\Q$, we recall the dual lattice $\OK^\#$ of the lattice of integers $\OK$, also known as the inverse different ideal of $E$, which in our case can be explicitly written as $\OK^\# = \frac{1}{\sqrt{D_E}} \OK$, where $D_E (< 0)$ is the discriminant of $E$. We say a matrix $m$ is semi-integral (over $E$) if the diagonal entries of $m$ are in the ring $\OK$ and the off-diagonal entries  are in the dual lattice $\OK^\#$. 

Let $k$ be an integer, and let $m$ be an $l \times l$ semi-integral Hermitian matrix. Let $\rho$ be a complex finite dimensional representation of $\Ggl$ under the embedding \eqref{kiwi}. We say a holomorphic vector-valued function $\phi : \Hgl \lra V (\rho)$ is a Hermitian Jacobi form of genus $g$, weight $k$, index $m$, and type $\rho$, if the function $f: \bH_{g + l} \lra V (\rho)$ defined via
\begin{linenomath}
\begin{gather*}
    f \Bigg(
    \begin{pmatrix}
    \tau & w\\
    z & \tau^\prime
    \end{pmatrix}
    \Bigg)
    =
    \phi (\tau, w, z)
    e (m \tau^\prime)
\end{gather*}
\end{linenomath}
transforms as a Hermitian modular form for the image of \eqref{kiwi}, of degree $g + l$, weight $k$ and type $\rho$, and if $g = 1$, the function $\phi (\tau, \tau r + r^\prime, s \tau + s^\prime)$ also being bounded (under any norm as $\dim V(\rho) < \infty$) at the cusp for all row vectors $r, r^\prime$ and column vectors $s, s^\prime$ of dimension $l$ over $E$.  The space of Hermitian Jacobi forms of genus $g$, weight $k$, index $m$, and type $\rho$ is denoted by $\JR$ in this manuscript.

\subsection{Fourier--Jacobi expansions}\label{placidity}
Let $g, l$ be integers such that $1 \leq l \leq g - 1$. Fourier--Jacobi expansions of Hermitian modular forms are partial Fourier expansions with respect to the matrix block $\tau_2$ of the variable $\tau = \begin{psmatrix}
\tau_1 & z\\
w & \tau_2
\end{psmatrix}$.
More precisely, for a Hermitian modular form $f$ of degree $g$, weight $k$, and type $\rho$, there is a Fourier expansion (\cite{braun-1949}) in the form of
\begin{linenomath}
\begin{gather}\label{blackberry}
    f (\tau) = 
\sum_{t \in \HGK_{\geq 0}} c (f; t) e (t \tau)
\tx{,}
\end{gather}
\end{linenomath}
for $c (f; t) \in V(\rho)$. In particular, when $\rho$ is the trivial representation, the sum is supported on semi-integral positive semidefinite Hermitian matrices $t$. We write
\begin{linenomath}
 \begin{gather*}
     \tau = \begin{psmatrix}
\tau_1 & w\\
z & \tau_2
\end{psmatrix} \in \Hg
\quad \tx{and} \quad
t = \begin{psmatrix}
n & r\\
r^\ast & m
\end{psmatrix} \in \HGK_{\geq 0}
 \end{gather*}
\end{linenomath}
for $\tau_1 \in \bH_{g - l}, \tau_2 \in \bH_{l}, w \in  \Mat{g - l, l} (\C), z \in \Mat{l, g - l} (\C), n \in \mathrm{Herm}_{g - l} (E)_{\geq 0}, r \in \mathrm{Mat}_{g - l, l} (E)$, and $m \in \HLK_{\geq 0}$. Consequently, we arrange the sum in \eqref{blackberry} into the form
\begin{linenomath}
\begin{gather}\label{apricot}
   f (\tau) = 
\sum_{m \in \HLK_{\geq 0}} \phi_m (\tau_1, w, z) e (m \tau_2)
\end{gather}
\end{linenomath}
for functions $\phi_m : \mathbb{H}_{g - l, l} \lra V(\rho)$ with the ordinary Fourier expansion
\begin{linenomath}
\begin{gather}\label{lemon} 
    \phi_m (\tau_1, w, z)
    =
    \sum_{\substack{n \in \mathrm{Herm}_{g - l} (E)_{\geq 0}\\
    r \in \mathrm{Mat}_{g - l, l} (E)}}
    c (\phi_m; n, r) e (n \tau_1 + r z) e (r^\ast w)
    \tx{,}
\end{gather}
\end{linenomath}
whose Fourier coefficients $c (\phi_m; n, r)$ are
determined by the Fourier coefficients of $f$ via the equation
\begin{linenomath}
\begin{gather}\label{han}
    c (\phi_m; n, r)
    =
    c \Bigg(f; \begin{pmatrix}
    n & r \\
    r^\ast &  m
    \end{pmatrix}  \Bigg)
    \tx{.}
\end{gather}
\end{linenomath}
We say Equation~\eqref{apricot} is the Fourier--Jacobi expansion of $f$ of cogenus $l$, and $\phi_m$ is the $m$-th Fourier--Jacobi coefficient of $f$, or the Fourier--Jacobi coefficient of index $m$. Furthermore, it follows from the definition that any Jacobi form $\phi \in \JR$ has the Fourier expansion in the form of \eqref{lemon}.

\subsection{Symmetric formal Fourier--Jacobi series}\label{calmness}
We motivate in this subsection the notion of symmetric formal Fourier--Jacobi series for the Hermitian modular group, by combining two features of the Fourier--Jacobi expansion of a Hermitian modular form: each Fourier--Jacobi coefficient is a Hermitian Jacobi form, and they satisfy a $\GL{g} (\OK)$-symmetry condition. A priori, for formal series these two features ``almost'' characterize ``formal modularity'', that is, invariance under the action of Hermitian modular group without assuming absolute convergence. Indeed, if $\OK$ is a PID, the group-theoretic Lemma~\ref{toy} amounts to saying that formal modularity can be reformulated in terms of two conditions: invariance under the embedded Jacobi group of cogenus $1$, and invariance under the embedded subgroup $\rot \big(\GL{g} (\OK) \big)$.

It is clear that the invariance of a formal Fourier series $f$ in the form of \eqref{blackberry} under the weight $k$, type $\rho$-slash action of $\rot \big(\GL{g} (\OK)\big)$ is equivalent to a symmetry condition for the formal Fourier coefficients of $f$, which reads
\begin{linenomath}
\begin{gather}\label{zhou}
    \rho \Bigg(
    \begin{pmatrix}
u & 0 \\
0 & {u^\ast}^{-1}
\end{pmatrix}
    \Bigg)
    c (f; u^\ast t u)
    = \big(\det u^\ast \big)^k 
    c (f; t)
\end{gather}
\end{linenomath}
for all $u \in \GL{g} (\OK)$.

Moreover, the invariance of $f$ under the weight $k$, type $\rho$-slash action of the embedded Jacobi group $\Ggl$ under \eqref{kiwi} amounts to saying that $f$ can be rearranged into a formal series in the form of \eqref{apricot}, such that each coefficient $\phi_m$ is a formal Hermitian Jacobi form of genus $g - l$, weight $k$, index $k$, and type $\rho$.

Combining these two aspects, we define a symmetric formal Fourier--Jacobi series $f$ of degree $g$, cogenus $l$, weight $k$, and type $\rho$, to be a formal series of Hermitian Jacobi forms $\phi_m \in \mathrm{J}_{k, m}^{(g - l)} (\rho)$ in the form of \eqref{apricot}, such that its Fourier coefficients $c (f; t)$, defined by the equation
\begin{linenomath}
\begin{gather*}
    c \Bigg(f; \begin{pmatrix}
    n & r \\
    r^\ast &  m
    \end{pmatrix}  \Bigg)
    =
     c (\phi_m; n, r)
     \tx{,}
\end{gather*}
\end{linenomath}
satisfy the symmetry condition \eqref{zhou}. 

We write $\FMR$ for the vector space of such symmetric formal Fourier--Jacobi series, and when the cogenus $l = 1$ or $\rho$ is the $1$-dimensional trivial representation, we suppress them in the notation. Furthermore, we let
\begin{linenomath}
\begin{gather*}
    \FMG (\rho) := \bigoplus_{k \in \Z} \FMAA
\end{gather*}
\end{linenomath}
denote the graded module over the graded ring $\MGR$ of classical modular forms of type $\rho$, and note that $\FMG$ is actually a graded algebra over $\MG$. 
\begin{remark}
Of course, the readers will not confuse the degree $g$ in $\FMR$ for the genus $g$ in $\JR$. In fact, for a symmetric formal Fourier--Jacobi series $f \in \FMR$, the $m$-th coefficient $\phi_m$ of $f$ is a Jacobi form of genus $g - l$ and index $m$, more precisely $\phi_m \in \mathrm{J}_{k, m}^{(g - l)} (\rho)$.
\end{remark}

\subsection{Formal Fourier--Jacobi coefficients}\label{Alice}
Just like Fourier--Jacobi coefficients of arbitrary cogenus are attached to a Hermitian modular form, we also associate formal Fourier--Jacobi coefficients of various cogenus to a formal Fourier--Jacobi series. 

Let $1 \leq l^\prime < l < g$ be positive integers, and let $f \in \MR$ be a Hermitian modular form. For $\tau = \begin{psmatrix} \tau_1 & w \\ z & \tau_2
\end{psmatrix} \in \Hg$, we refine the decomposition of the matrix $\tau$ further into blocks
\begin{linenomath}
\begin{gather*}
    \tau = \begin{pmatrix} \tau_{11} & w_{11} & w_{12}\\
    z_{11} & \tau_{12} & w_{22}\\
    z_{12} & z_{22} & \tau_2
    \end{pmatrix} 
    \tx{,}
\end{gather*}
\end{linenomath}
so that the sizes of $\tau_{11}, \tau_{12}$, and $\tau_2$ are $(g - l) \times (g - l), (l - l^\prime) \times (l - l^\prime)$, and $l^\prime \times l^\prime$, respectively, and the sizes of the off-diagonal blocks are determined in the ordinary way. We write $w_1$ for the $(g - l)\times l$ matrix $(w_{11}\, w_{12})$, and $z_1$ for the $l \times (g - l)$ matrix $\begin{psmatrix}z_{11} \\ z_{12}\end{psmatrix}$.

For a formal Fourier series $f$ in the form of \eqref{blackberry}, the formal Fourier--Jacobi expansion of cogenus $l$ is given by \eqref{apricot}, with the $m$-th formal Fourier--Jacobi coefficient $\psi_m$ defined in \eqref{lemon} and \eqref{han}, where we replace the notation $\phi_m$ by $\psi_m$, to emphasize it is a formal series. In particular, for a symmetric formal Fourier--Jacobi series $f \in \FMR$ with Fourier--Jacobi coefficients $\phi_m$, and for each Hermitian matrix $m^\prime \in \mathrm{Herm}_{l^\prime} (E)_{\geq 0}$, the formal Fourier--Jacobi coefficient $\psi_{m^\prime}$ of index $m^\prime$ is related to $\phi_m$ via an equation of formal Fourier-series, namely
\begin{linenomath}
\begin{gather}\label{John} 
    \psi_{m^\prime} (\tau_1, w, z)
    =
    \sum_{\substack{n^\prime \in \mathrm{Herm}_{l - l^\prime} (E)_{\geq 0}\\
    r^\prime \in \mathrm{Mat}_{l - l^\prime, l^\prime } (E)}}
    \phi_{\begin{psmatrix} n^\prime & r^\prime \\ {r^\prime}^\ast & m^\prime \end{psmatrix}}
    (\tau_{11}, w_1, z_1)
    e (n^\prime \tau_{12} + r^\prime z_{22}) e ({r^\prime}^\ast w_{22})
    \tx{.}
\end{gather}
\end{linenomath}
In particular, if $l^\prime = 1$, we record the following fact about the formal Fourier--Jacobi coefficient $\psi_0$ of index $0 \in \mathrm{Herm}_{1} (E)_{\geq 0}$.

\begin{proposition}\label{finnegan}
Let $f \in \FMC$ be a symmetric formal Fourier--Jacobi series of trivial type. Then the formal Fourier--Jacobi coefficient $\psi_0$ defined by Equation \eqref{John} is a symmetric Fourier–Jacobi series of degree $g - 1$ and cogenus $l - 1$ (under an identification), namely $\psi_0 \in \mathrm{FM}_k^{(g - 1, l - 1)}$.
\end{proposition}
\begin{proof}
For any Hermitian matrix $\begin{psmatrix} n^\prime & r^\prime \\ {r^\prime}^\ast & 0 \end{psmatrix}$ that occurs in Equation \eqref{John}, it is positive semidefinite, hence $r^\prime = 0$. Moreover, for any index $n^\prime \in \mathrm{Herm}_{l - 1} (E)_{\geq 0}$, as $\phi_{\begin{psmatrix} n^\prime & 0 \\ 0 & 0 \end{psmatrix}}$ is a Hermitian Jacobi form,  by the transformation law it must be a constant function in $w_{12}$ and $z_{12}$. In particular, $\phi_{\begin{psmatrix} n^\prime & 0 \\ 0 & 0 \end{psmatrix}}$ can be identified as a Jacobi form of genus $g - l$ and index $n^\prime$, which we write as $\phi_{n^\prime}$. Therefore, $\psi_0$ can be identified as a formal Fourier--Jacobi series
\begin{linenomath}
\begin{gather*}
    \psi_0 (\tau_1) = 
    \sum_{n^\prime \in \mathrm{Herm}_{l - 1} (E)_{\geq 0}  }
    \phi_{n^\prime} (\tau_{11}, w_{11}, z_{11})
    e (n^\prime \tau_{12}) 
    \tx{.}
\end{gather*}
\end{linenomath}
The symmetry condition of the formal Fourier--Jacobi series $\psi_0$ under the action of the embedded subgroup $\GL{g - 1} (\OK)$ follows from the one of $f$ under the action of $\GL{g} (\OK)$.
\end{proof}

\subsection{Formal theta decomposition}\label{whispering}
We show in this subsection that the classical theta decomposition of Jacobi forms has a formal analogue, namely the formal theta decomposition of formal Fourier--Jacobi coefficients.

Let $g, l$ be arbitrary positive integers, and let $m \in \mathrm{Herm}_l (E)_{> 0}$ be a semi-integral positive definite Hermitian matrix, which we also view as an integral positive definite quadratic form. Let $\Delta_{g} (m)$ denote the finite abelian group
\begin{linenomath}
 \begin{gather*}
     \Delta_{g} (m) = \MglA ({\OK}^{\#}) / \MglZ m
     \tx{,}
 \end{gather*}
\end{linenomath}
and let $\rho_m^{(g)}$ be the Weil representation of $\UGZ$ associated with $\Delta_{g} (m)$ on $\C [\Delta_{g} (m)]$. Recall that
the theta series $\theta_{m, s}^{(g)}: \Hgl \lra \C$ of genus $g$, integral quadratic form $m$, and shift $s \in \Delta_{g} (m)$, is defined by \begin{linenomath}
\begin{gather}\label{bright}
    \theta_{m, s}^{(g)} (\tau, w, z) 
    = \sum_{r \in s + \MglZ m}
    e (r m^{-1} r^\ast \tau + r z) e(r^\ast w)
    \tx{.}
\end{gather}
\end{linenomath}

It is a classical result (Proposition 3.5 of \cite{shimura-1978a}; see also Theorem 5.1 of \cite{eichler-zagier-1985}) that the vector-valued theta series ${\big(\theta_{m, s}^{(g)}\big)}_{s \in \Delta_{g} (m)}$ transforms as a vector-valued Hermitian Jacobi form of weight $l$ and type ${\big(\rho_m^{(g)}\big)}^{\vee}$. Furthermore, any Hermitian Jacobi form $\phi \in \JA$ can be  uniquely written as a sum
\begin{linenomath}
\begin{gather}\label{Nadal}
\phi (\tau, w, z) = 
\sum_{s \in \Delta_{g} (m)}
h_s (\tau)  \theta_{m, s}^{(g)} (\tau, w, z)
\end{gather}
\end{linenomath}
for some holomorphic functions $h_s$, and these functions are the components of a vector-valued Hermitian modular form of type $\rho_m^{(g)}$. 

Let $1 \leq l^\prime < l < g$ be positive integers. For a symmetric formal Fourier--Jacobi series $f \in \FMC$ with Fourier--Jacobi coefficients $\phi_m$ for $m \in \mathrm{Herm}_{l}(E)_{\geq 0}$, we construct the formal theta decomposition of its formal Fourier--Jacobi coefficients. Let $m^\prime \in \mathrm{Herm}_{l^\prime} (E)_{\geq 0}$ be a positive semidefinite Hermitian matrix, and let $\psi_{m^\prime}$ be the $m^\prime$-th formal Fourier–Jacobi coefficient of $f$. Since the formal Fourier coefficients  of $f$ satisfy by definition the symmetry condition~\eqref{zhou}, setting $u = \begin{psmatrix}I_{g - l^\prime} & 0 \\
\lambda^\ast & I_{l^\prime}
\end{psmatrix}$ for $\lambda \in \mathrm{Mat}_{g - l^\prime, l^\prime} (\OK)$ in \eqref{zhou}, we find a symmetry condition for the formal Fourier coefficients of $\psi_{m^\prime}$, namely
\begin{linenomath}
\begin{gather}\label{Bob}
    c (\psi_{m^\prime}; n^\prime + r^\prime {m^\prime}^{-1} {r^\prime}^\ast, r^\prime)
    =
    c (\psi_{m^\prime}; n^\prime + r^{\prime\prime} {m^\prime}^{-1} {r^{\prime\prime}}^\ast, r^{\prime\prime})
\end{gather}
\end{linenomath}
for all $n^\prime \in \mathrm{Herm}_{g - l^\prime} (E)_{\geq 0}$ and all $r^\prime, r^{\prime\prime} \in \mathrm{Mat}_{g - l^\prime, l^\prime} (E)$ such that $r^{\prime\prime} = r^\prime + \lambda m^\prime$ for some  $\lambda \in \mathrm{Mat}_{g - l^\prime, l^\prime} (\OK)$. In particular, for each $s^\prime \in \Delta_{g - l^\prime} (m^\prime)$, Equation~\eqref{Bob} defines a formal Fourier series $h_{m^\prime, s^\prime}$ in $\tau_1$ via its formal Fourier coefficients
\begin{linenomath}
\begin{gather}\label{Caesar}
    c (h_{m^\prime, s^\prime}; n^\prime) :=
    c (\psi_{m^\prime}; n^\prime + r^\prime {m^\prime}^{-1} {r^\prime}^\ast, r^\prime)
    \tx{,}
\end{gather}
\end{linenomath}
for any choice of $r^\prime \in s^\prime + \mathrm{Mat}_{g - l^\prime, l^\prime} (\OK) m^\prime$ and all $n^\prime \in \mathrm{Herm}_{g - l^\prime} (E)_{\geq 0}$. It then follows that
\begin{linenomath}
\begin{gather}\label{Celine}
\psi_{m^\prime} (\tau_1, w, z) = 
\sum_{s^\prime \in \Delta_{g - l^\prime} (m^\prime)}
h_{m^\prime, s^\prime} (\tau_1)  \theta_{m^\prime, s^\prime}^{(g - l^\prime)} (\tau_1, w, z)
\end{gather}
\end{linenomath}
holds as an equation of formal Fourier series.

Our aim is to show that the vector-valued formal Fourier series $(h_{m^\prime, s^\prime})_{s^\prime}$ is in fact a symmetric formal Fourier--Jacobi series. This fact can be proved for arbitrary cogenus $l^\prime$; for simplicity we only do the case $l^\prime = l - 1$ used in Section~\ref{jump}.  We separate our proof into the following two lemmas, and state the conclusion after that.
\begin{lemma}\label{Nathan}
Fix $l^\prime = l - 1$. Then, the formal vector-valued Fourier series $h_{m^\prime} = (h_{m^\prime, s^\prime})_{s^\prime \in \Delta_{g - l^\prime} (m^\prime)}$ defined by Equation~\eqref{Caesar} satisfies the symmetry condition \eqref{zhou} for the Weil representation $\rho = \rho_{m^\prime}^{(g - l^\prime)}$.
\end{lemma}
\begin{proof}
Setting $u = \begin{psmatrix}
u^\prime & 0\\
0 & I_{l^\prime}
\end{psmatrix}$ for $u^\prime \in \GL{g - l^\prime} (\OK)$ in the symmetry condition~\eqref{zhou} of $f \in \FMC$, we find a symmetry condition
\begin{linenomath}
\begin{gather}\label{Leon}
    c (\psi_{m^\prime}; {u^\prime}^\ast n^\prime u^\prime, {u^\prime}^\ast r^\prime)
    =
    (\det {u^\prime}^\ast)^k
    c (\psi_{m^\prime}; n^\prime, r^\prime)
\end{gather}
\end{linenomath}
for all $n^\prime \in \mathrm{Herm}_{g - l^\prime} (E)_{\geq 0}$, $r^\prime \in \mathrm{Mat}_{g - l^\prime, l^\prime} (\OK)$ , and $u^\prime \in \GL{g - l^\prime} (\OK)$. Combining the definition \eqref{Caesar} of $h_{m^\prime, s^\prime}$ and Equation~\eqref{Leon}, we see that 
\begin{linenomath}
\begin{gather*}
    c (h_{m^\prime, {u^\prime}^\ast s^\prime}; {u^\prime}^\ast n^\prime u^\prime)
    =
    (\det {u^\prime}^\ast)^k
    c (h_{m^\prime, s^\prime}; n^\prime)
\end{gather*}
\end{linenomath}
holds for all $n^\prime \in \mathrm{Herm}_{g - l^\prime} (E)_{\geq 0}$, $s^\prime \in \Delta_{g - l^\prime} (m^\prime)$, and $u^\prime \in \GL{g - l^\prime} (\OK)$, as desired.
\end{proof}
Next, for the formal Fourier series $h_{m^\prime, s^\prime}$ defined by \eqref{Caesar}, we consider their cogenus-$1$ formal Fourier--Jacobi expansions
\begin{linenomath}
\begin{gather}\label{Sasha}
    h_{m^\prime, s^\prime}(\tau_1)
    =
    \sum_{n \in \Q_{\geq 0}}
    \psi_{m^\prime, s^\prime, n} (\tau_{11}, w_{11}, z_{11})
    e(n \tau_{12})
    \tx{,}
\end{gather}
\end{linenomath}
and relate their formal Fourier--Jacobi coefficients $\psi_{m^\prime, s^\prime, n}$ to the Fourier--Jacobi coefficients $\phi_m$ of $f \in \FMC$. In fact, Equations~\eqref{John} and \eqref{Celine} provide such a relation. Inserting Equation~\eqref{Sasha} in Equation~\eqref{Celine}, and writing $s^\prime \in \Delta_{g - l^\prime} (m^\prime)$ in the form $\begin{psmatrix}s_1^\prime \\ s_2^\prime\end{psmatrix}$ for 
$s_1^\prime \in \Delta_{g - l} (m^\prime)$ and $s_2^\prime \in \Delta_{1} (m^\prime)$, we obtain the following result after comparing the formal Fourier coefficient at $e (n^\prime \tau_{12} + r^\prime z_{22}) e({r^\prime}^\ast w_{22})$ with Equation~\eqref{John}, and simplifying the expression.
\begin{lemma}\label{Anastasia}
Let Equation~\eqref{Sasha} be the formal Fourier--Jacobi expansion of cogenus $1$ of $h_{m^\prime, s^\prime}$. Then, for any shift $s_2^\prime \in \Delta_{1} (m^\prime)$, every representative $r^\prime \in s_2^\prime + \mathrm{Mat}_{1, l^\prime} (\OK) m^\prime$, and every integer $n^\prime \geq 0$, the partial theta decomposition of Fourier--Jacobi coefficients
\begin{linenomath}
\begin{align}\label{Misha}
    &\phi_{\begin{psmatrix}n^\prime & r^\prime \\
    {r^\prime}^\ast & m^\prime \end{psmatrix}} (\tau_{11}, w_1, z_1)
    =\\
    &\sum_{s_1^\prime \in \Delta_{g - l} (m^\prime)} 
    \psi_{m^\prime, \begin{psmatrix}s_1^\prime \\
    \label{Rita}
    s_2^\prime\end{psmatrix}, 
    n^\prime - r^\prime {m^\prime}^{-1} {r^\prime}^\ast}
    (\tau_{11}, w_{11}, z_{11})
    \theta_{m^\prime, s_1^\prime}^{(g - l)}
    (\tau_{11}, z_{12} + {m^\prime}^{-1} {r^\prime}^\ast z_{11},
    w_{12} + w_{11}  r^\prime {m^\prime}^{-1})
\end{align}
\end{linenomath}
holds as an equation of formal Fourier series. 
In particular, every formal Fourier--Jacobi coefficient $\psi_{m^\prime, s^\prime, n}$ in \eqref{Sasha} is a holomorphic function in $\tau_1$. 
\end{lemma}

\begin{proposition}\label{boutique}
Let $g > l$ be positive integers and fix $l^\prime = l - 1$. Let $f \in \FMC$ be a symmetric formal Fourier--Jacobi series, and let $m^\prime \in \mathrm{Herm}_{l^\prime} (E)_{> 0}$ be a semi-integral positive definite Hermitian matrix. Then, the formal Fourier–Jacobi coefficient $\psi_{m^\prime}$ of $f$ has a formal theta expansion
\begin{linenomath}
\begin{gather*}
    \psi_{m^\prime} (\tau_1, w, z)
= \sum_{s^\prime \in \Delta_{g - l^\prime} (m^\prime)}
h_{m^\prime, s^\prime} (\tau_1)  \theta_{m^\prime, s^\prime}^{(g - l^\prime)} (\tau_1, w, z)
\tx{,}
\end{gather*}
\end{linenomath}
where the coefficients are vector-valued symmetric formal Fourier--Jacobi series, more precisely
\begin{linenomath}
 \begin{gather*}
     (h_{m^\prime, s^\prime})_{s^\prime} \in \mathrm{FM}_{k - l^\prime}^{(g - l^\prime)} (\rho_{m^\prime}^{(g - l^\prime)})
     \tx{.}
 \end{gather*}
\end{linenomath}
\end{proposition}
\begin{proof}
By Lemmas~\ref{Nathan} and \ref{Anastasia}, it suffices to show the modularity of the vector-valued function
\begin{linenomath}
 \begin{gather*}
     (\psi_{m^\prime, s^\prime, n})_{s^\prime \in \Delta_{g - l^\prime} (m^\prime)}
     \tx{.}
 \end{gather*}
\end{linenomath}
We analyze both sides of the equation in Lemma~\ref{Anastasia}. On the left hand side \eqref{Misha}, the function is invariant under the weight-$k$ action of the Hermitian Jacobi group $\Gamma^{(g - l, l)}$ and index $\begin{psmatrix}n^\prime & r^\prime \\
{r^\prime}^\ast & m^\prime \end{psmatrix}$, by the assumption $f \in \FMC$. On the right hand side \eqref{Rita}, the vector-valued theta function transforms by the restriction of $(\rho_{m^\prime}^{(g)})^\vee$ to $\Gamma^{(g - l, l^\prime)}$, of weight $l^\prime$. Arguing as in Section 3 of \cite{ziegler-1989}, we conclude that the vector-valued holomorphic function $(\psi_{m^\prime, s^\prime, n})_{s^\prime \in \Delta_{g - l^\prime} (m^\prime)}$ transforms by the restriction of $\rho_{m^\prime}^{(g)}$ to $\Gamma^{(g - l, 1)}$, of weight $k - l^\prime$, which completes the proof.
\end{proof}

Finally, we state the classical result mentioned from the beginning. Let $\phi \in \JA(\rho)$ be a Hermitian Jacobi form, then its Fourier coefficients $c (\phi; n, r)$ satisfy a symmetry condition similar to Equation~\eqref{Bob}, which then defines a formal Fourier series $h_s$ via
\begin{linenomath}
\begin{gather*}
    c (h_s; n) := c (\phi; n + r m^{-1} r^\ast, r)
\end{gather*}
\end{linenomath}
for each $s \in \Delta_{g} (m)$ and any choice of $r \in s + \mathrm{Mat}_{g, l} (\OK) m$. Then we have an equation of formal Fourier series \eqref{Nadal}, and the vector-valued holomorphic function $(h_s)_{s \in\Delta_{g} (m)}$ satisfying Equation~\eqref{Nadal} is unique. Similar to the formal theta decomposition, we follow the lines of \cite{ziegler-1989}, Section 3 and obtain the classical theta decomposition in the Hermitian case.

\begin{proposition}\label{greece}
Let $g$ and $l$ be positive integers. Let $m \in \mathrm{Herm}_{l} (E)$ be a semi-integral Hermitian matrix. Let $\rho_m^{(g)}$ be the Weil representation of $\UGZ$ on $\C [\Delta_{g} (m)]$, and let $\rho$ be a complex finite dimensional representation of $\UGZ$. Then, there is an isomorphism \begin{linenomath}
\begin{align*}
    \JA(\rho) &\cong \mathrm{M}_{k - l}^{(g)} (\rho_m^{(g)} \otimes \rho)\\
    \phi &\lmto (h_s)_{s \in \Delta_{g} (m)}
    \tx{,}
\end{align*}
sending a Hermitian Jacobi form $\phi$ to the components of its theta decomposition.
\end{linenomath}
\end{proposition}

\section{Symmetric formal Fourier--Jacobi series are algebraic over Hermitian modular forms}\label{run}
The aim of this section is to show Theorem~\ref{lolita}. Throughout this section, let $E/\Q$ be a norm-Euclidean imaginary quadratic field,  over which the Hermitian modular forms are defined. In fact, except for Lemma~\ref{cell} and its corollaries, the whole section also works for $E/\Q$ being an arbitrary imaginary quadratic field. The proof of Theorem~\ref{lolita} is based on the following tension: on the one hand, the dimension formula in Proposition~\ref{dostoevsky} provides an asymptotic lower bound as in Corollary~\ref{cat} for the total dimension of Hermitian modular forms of bounded weight; on the other hand, the embedding constructed in Lemma~\ref{orange} gives rise to an asymptotic upper bound as in Corollary~\ref{motive} for the total dimension of symmetric formal Fourier--Jacobi series of bounded weight; but the orders of these two bounds in terms of the weight coincide, due to the nonvanishing of the lower slope bound for Hermitian modular forms, defined in Subsection~\ref{bible} and proven in Subsection~\ref{torch}.

\subsection{Vanishing orders and slope bounds}
\label{bible}
Let $t \in \HGK$, we say that $t$ represents a rational number $h$ if there is a nonzero integral column vector $\omega \in \cO_E^g \setminus \{0\}$ such that $\omega^\ast t \omega = h$. For a nonzero (vector valued) function $f$ with Fourier expansion as in \eqref{blackberry},  we define its vanishing order to be
\begin{linenomath}
\begin{gather*}
    \ord f:=
    \inf \big\{h \in \Q : h\, \text{can be represented by some}\, t \in \HGK\, \text{such that}\, c (f; t) \neq 0\big\}
    \tx{.}
\end{gather*}
\end{linenomath}
If $f = 0$, we set $\ord f := \infty$ by convention. 
\begin{lemma}\label{fig}
Let $f_1, f_2$ be two (vector-valued) functions.  Then their vanishing orders satisfy
\[\ord \big(f_1 \otimes f_2 \big) \geq \ord f_1 + \ord f_2
\tx{.}\]
\end{lemma}
\begin{proof}
The assertion is clear if $f_1 = 0$ or $f_2 = 0$. Assume now $f_1$ and $f_2$ are both nonzero, and let $t \in \HGK$ be a Hermitian matrix such that $c (f_1 \otimes f_2; t) \neq 0$. Since
\[c (f_1 \otimes f_2; t) = \sum_{\substack {0 \leq t_1, t_2 \in \HGK \tx{,} \\ t_1 + t_2 = t}} c (f_1 ; t_1) \otimes 
c(f_2; t_2)\tx{,}\]
there must be some pair $(t_1, t_2)$ such that $t_1 + t_2 = t$, $c (f_1; t_1) \neq 0$, and $c (f_2 ; t_2) \neq 0$, as desired.
\end{proof}

Similarly, we define the vanishing order of a nonzero Hermitian Jacobi form $\phi$ of genus $g$, weight $k$, index $m$ ($l \times l$), and type $\rho$  with Fourier expansion as in \eqref{lemon}, to be
\begin{linenomath}
\begin{align*}
    \ord\, \phi:=
    \inf \big\{h \in \Q : h\, \text{can be represented by some}&\, t \in \HGK\, \text{such that}\\
    & c (\phi; t, r) \neq 0 \,\, \tx{for some}\, r \in \Mgl  \big\}
    \tx{.}
\end{align*}
\end{linenomath}
If $\phi = 0$, we set $\ord\, \phi := \infty$ by convention.  Spaces of (vector-valued) Hermitian modular (resp. Hermitian Jacobi) forms  of genus $g$, weight $k$, (index $m$,) and type $\rho$ with vanishing order at least $o \in \Q$, are denoted by $\MR [o]$ (resp. $\JR [o]$).

Let $f$ be a classical (scalar-valued) Hermitian modular form of weight $k$. The slope of $f$ is defined as
\[\omega(f) := \frac{k}{\ord f}\tx{.}\]
The lower slope bound for classical Hermitian modular forms of degree $g$ is defined as
\[\omega_g := \inf_{\substack{k\in \Z\tx{,} \\ f \in \MA \setminus \{0\}}}
\omega (f)\tx{.}\]

\subsection{Embedding and vanishing}
\label{torch}

\begin{lemma}\label{lime}
Let $g \geq 2$ be an integer, and let $d, l$ be integers such that $1 \leq l \leq g - 1$. For every complex finite dimensional representation $\rho$ of $\UGZ$, which via restriction also denotes a type for the Jacobi group $\Gamma^{(g - l, l)}$, the linear map 
\begin{align*}
    \MA (\rho) [d] &\lra \FMC (\rho) [d] \\
    f & \lmto \sum_{0 \leq m \in \HLK}
    \phi_m (\tau_1, w, z) e (m \tau_2)
\end{align*}
given by the cogenus-$l$ Fourier--Jacobi expansion, is an embedding.
\end{lemma}
\begin{proof}
This is clear from the definition of Fourier--Jacobi expansion.
\end{proof}

\begin{lemma}\label{orange}
Let $g \geq 2$ be an integer. For every non-negative integer $d$ and every complex finite dimensional representation $\rho$ of $\UGZ$, we have
\[\dim \FMA (\rho) [d] \leq \sum_{m = d}^{\infty} \dim \JB (\rho) [m]\tx{.}\]
\end{lemma}

\begin{proof}
For any $f  \in \FMA (\rho) [d]$, by picking a non-canonical linear section
\[l_m: \FMA (\rho) [m] / \FMA (\rho) [m + 1] \lra  \FMA (\rho) [m]\]
for each $m \geq d$, we define recursively $f_m \in \FMA (\rho) [m]$ via $f_d := f$ and
\begin{linenomath}
 \begin{gather*}
     f_m := f_{m - 1} - l_{m - 1} ([f_{m - 1}])
     \tx{.}
 \end{gather*}
\end{linenomath}
Then we define a linear map
\begin{align*}
    i: \FMA (\rho) [d] &\lra \prod_{d \leq m \in \Z} 
    \FMA (\rho) [m] / \FMA (\rho) [m + 1]\\
    f &\lmto \big([f_m]\big)_{m \geq d}
    \,
    \tx{.}
\end{align*}
Since $[f_m] = [0]$ implies $l_m ([f_m]) = 0$ for each $m$, we have $f = f_n \in  \FMA (\rho) [n]$ for any $n \geq d$, hence $f = 0$ and $i$ is an embedding. Moreover, by sending $f \in \FMA (\rho) [m]$ to its $m$-th Fourier--Jacobi coefficient, we obtain a linear map
\begin{gather*}
    \FMA (\rho) [m] \lra \JB (\rho) [m]
\end{gather*}
whose kernel is $\FMA (\rho) [m + 1]$,  hence it induces an embedding \begin{gather*}
    j: \prod_{d \leq m \in \Z} 
    \FMA (\rho) [m] / \FMA (\rho) [m + 1] 
    \lhra
    \prod_{d \leq m \in \Z}
    \JB (\rho) [m]
    \tx{.}
\end{gather*}
The desired inequality follows by composing $i$, $j$ and the dimension function.
\end{proof}

\begin{proposition}\label{height}
Let $g$ be a positive integer, and let $\rho$ be a complex finite dimensional representation of $\UGZ$ that factors through a finite quotient. For every integer $o > \frac{k}{\omega_g}$, we have
\begin{linenomath}
\begin{gather*}
    \MR [o] = \{0\} \tx{.}
\end{gather*}
\end{linenomath}
\end{proposition}
\begin{proof}
If $\rho$ is the trivial representation of dimension $1$, the assertion is clear from the definition of $\omega_g$. In the general case, for any $f \in \MR [o]$ and $v \in V(\rho)^\vee$, we consider 
\begin{gather}\label{clementine}
    f_v := \prod_{\ga \in \ker (\rho) \backslash \UGZ} v \circ f|_k \ga\tx{.}
\end{gather}
It is clear that $f_v \in M_{|\rho| k}^{(g)}$, where $|\rho|$ is the index of $\ker (\rho)$ in  $\UGZ$. Note that each factor $v \circ f|_k \ga = v \circ \rho (\ga) f$ in \eqref{clementine} is a linear combination of components of $f$ as a vector-valued function, and $\ord f \geq o$, therefore $f_v \in M_{|\rho| k}^{(g)} \big[|\rho| o\big]$ by Lemma~\ref{fig}. But $|\rho| o > \frac{|\rho| k}{\omega_g}$ by the assumption $o > \frac{k}{\omega_g}$, as shown in the trivial case $M_{|\rho| k}^{(g)} \big[|\rho| o\big]$ vanishes, so $v \circ f|_k \ga = 0$ for some $\ga$. Taking the slash action of $\ga^{-1}$, we find $v \circ f = 0$, hence $f = 0$ as $v$ is chosen arbitrarily.
\end{proof}

Recall that there are exactly $5$ norm-Euclidean imaginary quadratic fields $E = \Q(\sqrt {d})$, which are given by $d \in \big\{-1, -2, -3, -7, -11\big\}$.  For each of these fields $E$, there is a maximal uniform constant $c_E > 0$, such that for any $\beta \in E$ there is an integer $\alpha \in \cO_E$ satisfying $\big(N_{E/\Q}(\beta - \alpha)\big)^2 \leq 1 - c_E$. In particular, we have the following corollary.

\begin{lemma}\label{cell} 
Let $s \in \Delta_g (m) = {\OK^\#}^g / m  \OK^g$. Then, there is a column vector $r \in s + m \OK^g$ such that its components $r_i$ satisfy $|r_i|^2 \leq (1 - c_E) m^2$.
\end{lemma}

Let $g$ be a positive integer and $r \in {\OK^\#}^g$ be a column vector of dimension $g$. In analogy to the ordinary vanishing order for Jacobi forms, we define the $r$-th vanishing order $\ord_r \phi$ to be the smallest $m \in \Q$ such that there is some Hermitian positive semidefinite matrix $t \in \mathrm{Herm}_{g} (E)_{\geq 0}$, satisfying $c (\phi; t, r) \neq 0$,  and the $(g, g)$-th entry $t_{g, g} = m$. It is clear that $\ord\, \phi \leq \ord_r \phi$ for every $r \in {\OK^\#}^g$, and the equation $\ord_r\, (f \psi) = \ord f + \ord_r \psi$ holds for every Hermitian modular form $f$ and Hermitian Jacobi form $\psi$ of the same degree.

\begin{proposition}\label{apple}
Let $g$ be a positive integer, and let $\rho$ be a complex finite dimensional representation of $\UGZ$ that factors through a finite quotient. For every integer $m > c_E^{-1} \frac{k}{\omega_g}$, we have
\[\JR [m] = 0\tx{.}\]
\end{proposition}
\begin{proof}
Let $\phi \in \JR [m]$ be a Jacobi form of cogenus-$1$. By Proposition~\ref{greece}, this Jacobi form $\phi$ has a unique theta decomposition 
\begin{linenomath}
\begin{gather}\label{car}
    \phi (\tau, w, z) = 
    \sum_{s \in \Delta_g (m) = {\OK^\#}^g / m  \OK^g}
    h_s (\tau)  \theta_{m, s}^{(g)} (\tau, w, z)
\tx{.}
\end{gather}
\end{linenomath}
For every $s \in \Delta_g (m)$ and every $r \in  s + m  \OK^g$, by Equation~\eqref{car} and the definition of the $r$-th vanishing order of Jacobi forms we have
\begin{linenomath}
\begin{gather}\label{sky}
    m \leq  \ord\, \phi \leq \ord_r \phi 
    = \ord_r    \big(h_s \theta_{m, s}^{(g)}\big)
    = \ord (h_s) + \ord_r \big(\theta_{m, s}^{(g)}\big)
    \tx{.}
\end{gather}
\end{linenomath}
On the other hand, by Lemma~\ref{cell}, for each $s \in \Delta_g (m)$, there is some column vector $r(s) \in  s + m  \OK^g$ whose entries $r_i$ satisfy $|r_i|^2 \leq (1 - c_E) m^2$ for each $i$. It then follows from the definition that $\ord_r (\theta_{m, s}^{(g)}) \leq (1 - c_E) m$ for $r = r(s)$, and hence $\ord (h_s) \geq c_E m$ by Inequality~\eqref{sky}. Since this holds for every $s$, by the assumption $m > c_E^{-1} \frac{k}{\omega_g}$ we have
\[h = (h_s)_{s\in \Delta_g (m)} 
\in \mathrm{M}_{k - 1}^{(g)} (\rho_m^{(g)} \otimes \rho) [o]\]
for some $o > \frac{k}{\omega_g} > \frac{k - 1}{\omega_g}$, which implies $h = 0$  by Proposition~\ref{height} and hence $\phi = 0$, as desired.
\end{proof}

\begin{corollary}\label{field}
For every positive integer $g$, the slope bound $\omega_g \geq 12 c_E^{g - 1} > 0$.
\end{corollary}
\begin{proof}
We prove by induction.
For $g = 1$ the upper half space $\bH_1$ is just the Poincar\'e upper half plane, and the modular group $\rmU(1, 1)(\Z)$ can be identified with $(\rmU(1) \cap \OK) \times \SL{2}(\Z)$. Hence $\rmM_k^{(1)}$ is just a subspace of the elliptic modular forms for $\SL{2} (\Z)$ of weight $k$, and it follows from \cite{bruinier-van-der-geer-harder-zagier-2008}, Page 9, Proposition 2, that $\omega_1 \geq 12$.

Now assume the assertion holds for $g = n - 1$ with $n \in \Z_{\geq 2}$, and we have to show $\omega_n \geq c_E \omega_{n - 1}$. Combining Lemma~\ref{lime}, Lemma~\ref{orange}, the induction hypothesis, and Proposition~\ref{apple}, we see that the inequalities
\[\dim \rmM_k^{(n)} [d] 
\leq \dim \mathrm{FM}_k^{(n)} [d] 
\leq \sum_{m = d}^{\infty}
\dim \rmJ_{k, m}^{(n - 1)}[m] 
= \sum_{m = d}^{\lfloor c_E^{-1} \frac{k}{\omega_{n - 1}} \rfloor} 
\dim \rmJ_{k, m}^{(n - 1)}[m]\]
hold for every positive integer $d$ and every weight $k$.
In particular, if $d > c_E^{-1} \frac{k}{\omega_{n - 1}}$, then the space $\MA [d]$ vanishes, which implies $\omega_n \geq c_E \omega_{n - 1}$. 
\end{proof}

\subsection{Asymptotic dimensions and algebraicity of $\FMG$ over $\MG$}\label{asymptotics}
\begin{lemma}\label{grass}
Let $\rho$ be a complex finite dimensional representation of $\rmU (1, 1) (\Z)$ that factors through a finite quotient, then
\[\dim \mathrm{M}_k^{(1)} (\rho) \ll \dim V(\rho) k \tx{.}\]
\end{lemma}
\begin{proof}
As $\mathrm{M}_k^{(1)} (\rho)$ is a subspace of elliptic modular forms  for $\SL{2} (\Z)$ of weight $k$ and some arithmetic type, the result follows from \cite{MR2114161}, Theorem 2.5.
\end{proof}

\begin{lemma}\label{glass}
Let $\rho$ be a complex finite dimensional representation of $\rmU (1, 1) (\Z)$ that factors through a finite quotient, then
\[\dim \mathrm{J}_{k, m}^{(1)} (\rho) \ll \dim V(\rho) k m^2 \tx{.}\]
\end{lemma}
\begin{proof}
Combining Proposition~\ref{greece} and Lemma~\ref{grass}, we obtain the desired asymptotic bound.
\end{proof}
\begin{remark}
The case of trivial type $\rho$  was also proved by Haverkamp in \cite{MR1377681}, Theorems 1 and 3.
\end{remark}

\begin{proposition}\label{class}
Let $g \geq 2$ be an integer and $\rho$ be a complex finite dimensional representation of $\UGZ$ that factors through a finite quotient, then
\[
\dim \FMAA \ll_g \dim V (\rho)  k^{g^2} \tx{.}
\]
\end{proposition}

\begin{proof}
We show the assertion by induction. For $g = 2$, Combining Lemma~\ref{orange}, Proposition~\ref{apple}, and Lemma~\ref{glass}, we have asymptotic inequalities
\begin{linenomath}
\begin{gather*}
    \dim \mathrm{FM}_k^{(2)} (\rho) 
    \leq  \sum_{m = 0}^{\lfloor 4 \frac{k}{\omega_{1}} \rfloor} \dim \mathrm{J}_{k, m}^{(1)} (\rho) [m] 
    \ll \dim V(\rho) k \sum_{m = 0}^{\lfloor 4 \frac{k}{\omega_{1}} \rfloor} m^2 
    \ll \dim V(\rho) k^4
    \tx{.}
\end{gather*}
\end{linenomath}
Assume the assertion is true for $g = n - 1$ with $n \geq 3$, and we show it for $g = n$.  Applying Lemma~\ref{orange}, Proposition~\ref{apple},  Proposition~\ref{greece}, and Lemma~\ref{lime} successively, we have \begin{linenomath}
\begin{gather}\label{book}
    \dim \FMN 
    \leq  \sum_{m = 0}^{\lfloor 4 \frac{k}{\omega_{n - 1}} \rfloor} \dim \JAA[m] 
    \leq \sum_{m = 0}^{\lfloor 4 \frac{k}{\omega_{n - 1}} \rfloor} \dim \mathrm{FM}_{k - 1}^{(n - 1)} (\rho_m^{(n - 1)} \otimes \rho)
    \tx{.}
\end{gather}
\end{linenomath}
The induction hypothesis and the rank of $\Delta_{n - 1} (m)$ for the Weil repsentation imply that
\begin{linenomath}
 \begin{gather*}
     \dim \mathrm{FM}_{k - 1}^{(n - 1)} (\rho_m^{(n - 1)} \otimes \rho) \ll_{n - 1} \dim V(\rho) m^{2 (n - 1)} k^{(n - 1)^2}
     \tx{,}
 \end{gather*}
\end{linenomath}
whence the desired assertion for $g = n$ together with \eqref{book} and Corollary~\ref{field}.
\end{proof}

\begin{corollary}\label{motive}
Let $g$ be a positive integer, then 
\[
\dim \FMK \ll_g k^{g^2 + 1}\tx{.}
\]
\end{corollary}

To estimate a lower bound for the dimension of each graded piece $\MA$, we recall the following Selberg trace formula applied to the dimensions of Hermitian modular cusp forms. We put a constant 
\[C (k, g) := 
2^{- g^2 - g} \pi^{- g^2} 
\prod_{0 \leq i, j \leq g - 1} 
(k - 2g + 1 + i + j)\tx{,}\] 
a discrete sum as a kernel function
\[
S(k, g, \tau):=
\sum_{\ga \in \UGZ / Z \big(\UGZ \big)}
\overline{j (\ga, \tau)}^{\, - k} \Big(\det \big(\frac{1}{2i} (\tau - (\ga\tau)^\ast ) \big) \Big)^{- k}\tx{,}\]
where  $Z \big(\UGZ \big)$ is the center of $\UGZ$,
and the (weight-$k$ normalized) hyperbolic volume form 
\[\rmd \tau = \Big(\det \big(\frac{1}{2i} (\tau - \tau^\ast) \big) \Big)^{k - 2g} \rmd_E \tau\]
where $\rmd_E \tau$ is the (complex) Euclidean measure .
\begin{proposition}[(\cite{MR871665}, Page 62)]\label{dostoevsky}
Let $k > 4 g - 2$ be an even integer. Then,
\[
   \dim S_k^{(g)}(\UGZ) =  C (k, g) \int_{Y_g}  S(k, g, \tau) \rmd \tau
   \tx{.}
\]
\end{proposition}

Observing the asymptotic behavior of $C(k, g)$ and the fact that $\int_{Y_g} S(k, g, \tau) \rmd \tau $ are bounded by some positive constant from below for sufficiently large even integer $k$, we deduce an asymptotic lower bound for the dimension of Hermitian modular forms $\MK$ of degree $g$ and weight at most $k$.

\begin{corollary}\label{cat}
Let $g$ be a positive integer, then 
\[
\dim \MK \gg_g k^{g^2 + 1}\tx{.}
\]
\end{corollary}

\begin{theorem}\label{lolita}
For every integer $g \geq 2$, the graded algebra $\FMG$ is an algebraic extension over the graded algebra $\MG$. 
\end{theorem}

\begin{proof}
It suffices to show that for an arbitrary $f \in \FMO$ of weight $k_0$, the set $\big\{f^i : i \in \N \big\}$ is of finite rank over the graded algebra $\MG$. Indeed, for any positive integers $k, n$ such that $k \geq n k_0$, by the fact that $\sum_{i = 0}^n \MK f^i  \subseteq \mathrm{FM}_{\leq k + n k_0}^{(g)}$, Corollary~\ref{motive}, and Corollary~\ref{cat}, we have
\[\dim_\C \big( \sum_{i = 0}^n \MK f^i \big)
\leq \dim_\C \mathrm{FM}_{\leq k + n k_0}^{(g)} 
\ll_g (k + n k_0)^{g^2 + 1}
\ll_g k^{g^2 + 1}
\ll_g \dim_\C \MK
\tx{,}\]
which implies the assertion.
\end{proof}

\section{Hermitian modular forms are algebraically closed in symmetric formal Fourier--Jacobi series}\label{fly}
In this section, we work in the case of an arbitrary imaginary quadratic field $E/\Q$, except Proposition~\ref{institute} which employs Lemma~\ref{toy}. For that part we assume $E$ is of class number $1$, in order to follow the simple definition of symmetric formal Fourier--Jacobi series throughout the manuscript. We prove Theorem~\ref{ulysses} based on the following strategy: in Subsection~\ref{roses}, we study the geometry of a toroidal boundary of the Hermitian modular variety and show that every prime divisor on the toroidal compactification intersects any open neighborhood of the toroidal boundary; we then recall the general construction of toroidal compactifications and the Weierstrass preparation theorem to show the local convergence of a symmetric formal Fourier--Jacobi series on an open neighborhood of the toroidal boundary in Subsection~\ref{lavender}; combining these two aspects, we conclude our argument in Subsection~\ref{irises} to show the global convergence of symmetric formal Fourier--Jacobi series. 

\subsection{Toroidal boundaries of Hermitian modular varieties}\label{roses}
The aim of this subsection is to show Theorem~\ref{crossing} based on the following strategy: the analogous result in the Siegel case for the Satake boundary, Lemma~\ref{lem:Jan_Martin_top}, was proven in the work of Bruinier--Raum; to reduce to the Siegel case, we embed the Siegel upper half space into the Hermitian upper half space both in a generic way and over $\Q$ via a density argument; finally we compare the behavior of various boundaries under the constructed maps. We also point out a cohomological approach to the theorem at the end of this subsection.

The theory of toroidal compactifications was developed by Ash--Mumford--Rapoport--Tai \cite{ash-mumford-rapoport-tai-1975}. A brief discussion of the construction of partial compacitifications of Siegel modular varieties can be found for instance in \cite{Hulek-Sakaran-2002} and \cite{bruinier-van-der-geer-harder-zagier-2008}.  Let $E/\Q$ be a fixed imaginary quadratic field  with an embedding $E \lhra \C$. The unitary group $\rmU (g, g) := \rmU_{E/\Q} (g, g)$ is a semi-simple algebraic group defined over $\Q$, and the group $\UGR^o$ is the connected component of $\mathrm{Aut} (\bD_g)$, for the Hermitian symmetric domain $\bD_g = \big\{\sigma \in \mathrm{Mat}_g (\C): \sigma^\ast \sigma < I_g \big\} \cong \bH_g$. For the arithmetric group $\GL{g} (\OK) \subseteq \UGR^o$, let $\{\sigma_\alpha\}_\alpha$ be a $\GL{g} (\OK)$-admissible collection of polyhedra. By Theorem 5.2 of \cite{ash-mumford-rapoport-tai-1975}, there is a unique toroidal compactification $X_g$ of the Hermitian modular variety $Y_g := \UGZ \backslash \Hg$ associated to this set of data, which is nonsingular in the orbifold sense. We write $\partial Y_g := X_g \setminus Y_g$ for the toroidal boundary.

Let $\overline{Y_g}$ be the Satake compactification of $Y_g$, and let 
\[\pi : X_g \lra \overline{Y_g}\]
be the natural map, which exists by the definition of the minimal compactification $\overline{Y_g}$. Let $\phi: \bH_g \lra Y_g$ be the quotient map. Recall that $\UGR$ acts transitively on the Hermitian upper half space $\bH_g$ via M\"obius transformations. In particular, for each $\ga \in \UGR$, we obtain an embedding of the Siegel upper half space $\cH_g$ into $\bH_g$ by sending $\tau \in \cH_g$ to $\ga \tau$, and the images $\ga \cH_g \subseteq \bH_g$ for $\ga \in \UGR$ cover $\bH_g$. Let $\cY_g :=  \Sp{g}(\bZ) \backslash \cH_g$ and  $\overline{\cY_g}$ denote the Siegel modular variety of degree $g$ and its Satake compactification, respectively.
\begin{lemma}\label{lem:density_rationality}
Let $D$ be a prime divisor on $Y_g$. Then, there is some rational point $\ga \in \UGQ$ such that
$\phi (\ga \cH_g)$ intersects $D$ transversally at some point. 
\end{lemma}
\begin{proof}
For any $\ga \in \UGR$, $\phi (\ga \cH_g)$ intersects $D$ transversally if and only if $\cH_g$ intersects $\ga^{-1} \phi^{-1} (D)$ transversally. It is a general fact that $\UGR$ acts transitively on the holomorphic unit tangent bundle of $\bH_g$ via $\ga (\tau, v) := (\ga \tau, \widehat{\rmd \ga_\tau v})$, where $\widehat{\rmd \ga_\tau v}$ denotes the unit vector associated to the pushforward of $v$ by the smooth map $\ga$ at $\tau$. Since $\phi^{-1} (D)$ is an embedded Hermitian submanifold of codimension $1$ in $\bH_g$, and M\"obius transformations are conformal with respect to the Hermitian metric, there is $\ga_0 \in \UGR$ such that $\cH_g$ intersects $\ga_0^{-1} \phi^{-1} (D)$ transversally at some point. In particular, there is an open neighborhood $W \subseteq \UGR$ of $\ga_0$ such that for all $\ga \in W$, $\cH_g$ intersects $\ga^{-1} \phi^{-1} (D)$ transversally at some point. The density of $\UGQ$ in $\UGR$ then implies the assertion.
\end{proof}

\begin{lemma}\label{lem:congruence_covering_immersion}
For every $\ga \in \UGQ$, there is a positive integer $N = N(\ga)$ such that the induced map
\begin{align*}
    [\ga]: \cY_g(N) = \Gamma(N) \backslash \cH_g &\lra Y_g = \UGZ \backslash \bH_g
\\
    [\tau] &\lmto [\ga \tau]
\end{align*}
is a local homeomorphism onto the image. Furthermore, the map can be continuously extended to
$\overline{[\ga]}: \overline{\cY_g (N)} \lra \overline{Y_g}$ on the Satake compactification, and for every open neighborhood $V$ of the Satake boundary $\overline{Y_g} \setminus Y_g$, there is an open neighborhood $U$ of the Satake boundary $\overline{\cY_g (N)} \setminus \cY_g (N)$, such that $\overline{[\ga]} (U) \subseteq V$.
\end{lemma}

\begin{proof}
As $\ga \in \UGQ$, there are positive integers $N_1, N_2$ such that $N_1 \ga$ and $N_2 \ga^{-1}$ are integral. Let $N := N_1 N_2$, then for any $\ga_N \in \Gamma (N)$, we have
\[\ga \ga_N \ga^{-1} - I_{2g} = 
\ga (\ga_N - I_{2g}) \ga^{-1} \in \ga N \Sp{g} (\bZ) \ga^{-1} \subseteq \UGZ\tx{,}\]
hence $\ga \Gamma (N) \ga^{-1} \subseteq \UGZ$. In other words, the map $[\ga]$ is well defined. It is clear that the induced surjective map onto the image is a covering map (of degree
$[\ga^{-1} \UGZ \ga \cap \Sp{g} (\bZ) : \Gamma(N)] $),
hence a local homeomorphism. Therefore, the continuous extension can be defined with respect to the topology of the Satake compactification. The boundary is mapped to the boundary under $\overline{[\ga]}$, hence $U := \overline{[\ga]}^{\, -1} (V)$ satisfies the desired property.
\end{proof}

\begin{lemma}\label{lem:level_covering}
For each positive integer $N$, the natural covering map
\begin{linenomath}
 \begin{gather*}
     \phi_N: \cY_g (N) = \Gamma(N) \backslash \cH_g \lra \cY_g =  \Sp{g}(\bZ) \backslash \cH_g
 \end{gather*}
\end{linenomath}
can be continuously extended to $\overline{\phi_N}: \overline{\cY_g (N)} \lra \overline{\cY_g}$  on the Satake compactification. For every open neighborhood $U$ of the Satake boundary  $\overline{\cY_g (N)} \setminus \cY_g (N)$, there is an open neighborhood $V$ of the Satake boundary $\overline{\cY_g} \setminus \cY_g$, such that $\overline{\phi_N}^{\, -1} (V) \subseteq U$.
\end{lemma}
\begin{proof}
Similar to Lemma~\ref{lem:congruence_covering_immersion}, the first assertion is clear. For the second one, since $\overline{\phi_N}$ maps the boundary to the boundary, it suffices to find an open neighborhood $V^\prime$ of the boundary in $\cY_g$ such that $\phi_N^{-1} (V^\prime) \subseteq U^\prime$ for $U^\prime = U \bigcap \cY_g (N)$.  Since $\phi_N$ is a covering map of finite degree $n = [\Sp{g} (\bZ) : \Gamma (N)]$, 
the restriction on the boundary neighborhood $\phi_N^{-1} \big( \phi_N (U^\prime) \big) \lra  \phi_N (U^\prime)$ is also a covering map of degree $n$.  Let $U_i^\prime$ be $i$-th sheet over $\phi_N (U^\prime)$ for $i = 1, 2, \ldots, n$. Then $V := V^\prime \bigcup (\overline{\cY_g} \setminus \cY_g)$ for $V^\prime := \bigcap_{i = 1}^{n} \phi_N (U^\prime \cap U^\prime_i)$ satisfies the desired property.
\end{proof}

\begin{lemma}\label{lem:Jan_Martin_top}
Assume $g \geq 2$, and let $D$ be a prime divisor on $\overline{\cY_g}$. Then $D$ intersects the Satake boundary $\overline{\cY_g} \setminus \cY_g$.
\end{lemma}
\begin{proof}
The assertion is shown in the second paragraph of the proof of Proposition~4.1 of \cite{MR3406827},  with $D$ replaced by $D^\prime$.
\end{proof}

\begin{theorem}\label{crossing}
Assume $g \geq 2$, and let $D$ be a prime divisor on $X_g$. Let $U \subseteq X_g$ be an open neighborhood of the boundary $\partial Y_g$. Then $D \cap U$ is a nontrivial divisor on $U$.
\end{theorem}

\begin{proof}
Recall the maps defined in this subsection:
\begin{gather*}
    X_g \overset{\pi} \longrightarrow 
    \overline{Y_g} \underset{\overline{[\ga]}} \longleftarrow
    \overline{\cY_g (N)} \overset{\overline{\phi_N}} \longrightarrow \overline{\cY_g}
\end{gather*}
for $\ga \in \UGQ$.
If $D$ is contained in the boundary $\partial Y_g = X_g \setminus Y_g$, the assertion is clear. Otherwise $D$ intersects $Y_g$, so the pushforward $D_1 := \pi_\ast (D) = \overline{\pi (D)}$ is a prime divisor on $\overline{Y_g}$, and $D_1^\prime := D_1 \cap Y_g$ is a prime divisor on  $Y_g$. 

By Lemma~\ref{lem:density_rationality}, there is $\ga \in \UGQ$, which we fix from now on, such that $\phi (\ga \cH_g)$ intersects $D_1^\prime$ transversally at some point. In particular, the intersection has codimension $1$ in $\phi (\ga \cH_g)$ at some point. Let $N := N (\ga)$, by Lemma~\ref{lem:congruence_covering_immersion}, $[\ga] : [\tau] \lmto \phi (\ga \tau)$ is a local homeomorphism onto the image. Moreover, since irreducible components are equidimensional, $\overline{[\ga]}^{-1} (D_1)$ contains a prime divisor $D_2$ on $\overline{\cY_g (N)}$. 

Since the Satake boundary has codimension at least $2$ for $g \geq 2$, $D_2$ must intersect $\cY_g (N)$. As $\phi_N$ is a covering map, hence a local homeomorphism, it follows that the pushforward $D_3 :=  \overline{\phi_N}_\ast (D_2)$ is a prime divisor on $\overline{\cY_g (N)}$. 

Finally we analyze the boundary intersection. Given an arbitrary open neighborhood $U \subseteq X_g$ of the toroidal boundary $\partial Y_g$, the open neighborhood $U_1 = \overline{Y_g} \setminus (X_g \setminus U) \subseteq \overline{Y_g}$ of the Satake boundary $\overline{Y_g} \setminus Y_g$ satisfies $\pi^{-1} (U_1) = U$. By Lemma~\ref{lem:congruence_covering_immersion}, there is an open neighborhood $U_2 \subseteq \overline{\cY_g (N)}$ of the boundary $\overline{\cY_g (N)} \setminus \cY_g (N)$ such that $\overline{[\ga]} (U_2) \subseteq U_1$. Finally, by Lemma~\ref{lem:level_covering}, there is an open neighborhood $U_3 \subseteq \overline{\cY_g}$ of the boundary $\overline{\cY_g} \setminus \cY_g$ such that $\pi^{-1} (U_3) \subseteq U_2$. By Lemma~\ref{lem:Jan_Martin_top}, $D_3$ intersects $U_3$, and it follows from the aforementioned choice of open neighborhoods that $D_2$ intersects $U_2$, $D_1$ intersects $U_1$, and $D$ intersects $U$, as desired.
\end{proof}

\begin{remark}\label{rmk:group_cohomology}
There is another approach to show Theorem~\ref{crossing} using results from stable cohomology, at least for large degree $g$. We sketch a proof here for reference. Following the strategy of proof of Proposition 4.1 in \cite{MR3406827}, it suffices to show the Picard group of $\overline{Y_g}$ has rank at most $1$. Since $\overline{Y_g} \setminus Y_g$ has codimension at least $2$ in $\overline{Y_g}$ for $g \geq 2$, we have $\mathrm{Pic} (\overline{Y_g}) = \mathrm{Pic} (Y_g)$. By the exponential sheaf sequence we have  ( \cite{hartshorne-1977} Appendix B) an exact sequence \begin{gather}\label{eq:exp_ext_seq}
    0 \lra H^1 (Y_g, \bZ) \lra H^1 (Y_g, \cO_{Y_g}) 
    \lra \mathrm{Pic} (Y_g) \lra H^2 (Y_g, \bZ) \lra \cdots
    \tx{.}
\end{gather}
Since $\bH_g$ is simply connected, $\pi_1 (Y_g) = 0$, hence $H^1 (Y_g, \cO_{Y_g}) = 0$ by general theory. By \eqref{eq:exp_ext_seq}, we only need to show $H^2 (Y_g, \bZ)$ has rank at most $1$. Applying results in \cite{borel-wallach-1980} (Chap. VII), this amounts to computing the rank for the degree-$2$ part in the graded ring $H^\ast (\UGZ, \bZ)$, which is canonically isomorphic to $H^\ast (A (\bH_g; \bZ)^{\UGZ} )$, where $A^q(M ; E)$ denotes the space of smooth $E$-valued differential $q$-forms on $M$ for $q \geq 0$. Alternatively, we can compute the lower degree parts of $H^\ast (\UGZ, \bZ)$ using Theorem 7.5, Table under 10.6 and Theorem 11.1 in \cite{MR387496}. In particular, we can apply these results to conclude Theorem~\ref{crossing} for sufficiently large $g$. For small $g$, one needs more explicit results of computing the second degree information in $H^\ast (A (\bH_g; \bZ)^{\UGZ} )$, which we skip here.
\end{remark}

\subsection{Local convergence at the boundary}
\label{lavender}
\begin{proposition}[(\cite{matsumura-1980}, Theorem 102)]
\label{Matsumura}
Let $k$ be a field of characteristic $0$, and $R$ be a regular ring containing $k$. Suppose that (1) for any maximal ideal $\frakm$ of $R$, the residue field $R/\frakm$ is algebraic over $k$ and $\mathrm{ht}\, \frakm = n$, and (2) there exist $D_1,\ldots, D_n \in \mathrm{Der}_k (R)$ and $x_1, \ldots, x_n \in R$ such that $D_i x_j = \delta_{ij}$. Then $R$ is excellent.
\end{proposition}
\begin{remark}
As pointed out in the remark after Theorem 102 in the book, convergent power series rings over $\C$ are examples of regular rings to which the theorem applies. 
\end{remark}

\begin{proposition}[(\cite{MR3406827}, Proposition 4.2)]
\label{diamond}
Let $A$ be a local integral domain with maximal ideal $\frakm$, and let $\hat{A}$ be the completion of $A$ with respect to $\frakm$. If $A$ is henselian and excellent, then $A$ is algebraically closed in $\hat{A}$.
\end{proposition}

\begin{proposition}\label{wind}
Let $f \in \FMA$ be a symmetric formal Fourier--Jacobi series of cogenus $1$. If $f$ is algebraic over the graded algebra $\MG$, then there is an open neighborhood $U$ of the boundary $\partial Y_g = X_g \setminus Y_g$ such that $f$ converges absolutely and defines a holomorphic function on $U$.
\end{proposition}
\begin{proof}
This is the Hermitian counterpart of Lemma 4.3 in \cite{bruinier-raum-2015}, and we follow the proof there. Let $x \in \partial Y_g$ be a toroidal boundary point, and let $Q \in \MG [X]$ be a polynomial  such that $Q (f) = 0$. Consider germs of holomorphic functions at the point $x$, and we write $\cO$ for the structure sheaf $\cO_{X_g}$. Note that the polynomial $Q$ defines a polynomial $Q_x \in \cO_x [X]$ and the symmetric formal Fourier--Jacobi series $f$ defines an element $f_x \in \hat{\cO}_x$ in the completion of the local ring $\cO_x$, such that $Q_x (f_x) = 0$. By general theory of toroidal compactifications and the Weierstrass preparation theorem, the local ring $\cO_x$ is henselian. On the other hand, Proposition~\ref{Matsumura} implies that the ring $\cO_x$ is excellent. Therefore, by Proposition~\ref{diamond} and the assumption that $f_x$ is algebraic over $\cO_x$, we conclude that $f$ converges absolutely in a neighborhood of $x$. Varying the boundary point $x$, we find $f$ converges in an open neighborhood of the whole toroidal boundary $\partial Y_g = X_g \setminus Y_g$.
\end{proof}

\subsection{Analytic continuation and algebraic closedness of $\MG$ in $\FMG$}
\label{irises}
\begin{proposition}[(\cite{MR3406827})]\label{wolf}
Let $N$ be a positive integer and $W \subseteq \C^N$ be a domain. Let $P(\tau, X)\\ \in \cO (W) [X]$ be a monic irreducible polynomial with discriminant $\Delta_P (\in \cO(W))$. Let $V \subseteq W$ be an open subset that intersects every irreducible component of the divisor $D = \mathrm{div} (\Delta_P)$. If $f$ is an analytic function on $V$ such that $P (\tau, f(\tau)) = 0$ on $V$, then $f$ has an analytic continuation to the entire domain $W$.
\end{proposition}

\begin{proposition}\label{institute}
Let $g \geq 2$ be an integer.
If $f \in \FMA$ is integral over the graded algebra $\MG$, then $f \in \MA$.
\end{proposition}
\begin{proof}
Let $P(\tau, X) \in \MG [X]$ be a monic irreducible polynomial over the graded algebra $\MG$ with coefficients of pure weights, such that $P (\tau, f(\tau)) = 0$ formally. By Proposition~\ref{wind}, there is an open neighborhood $U$ of the boundary $\partial Y_g = X_g \setminus Y_g$ such that $f$ converges absolutely and defines a holomorphic function on the preimage $V \subseteq \bH_g$ of $U$,  say $f|_V$.  By assumption, $f|_V$ satisfies $P (\tau, f|_V(\tau)) = 0$ on $V$.  Futhermore, the weights of the coefficients of $P$ form an arithmetic progression, hence $\Delta_P$ is a nonzero Hermitian modular form by the quasi-homogeneity of the discriminant. In particular,  $\mathrm{div} (\Delta_P)$ defines a divisor on $X_g$, hence every irreducible component of $\mathrm{div} (\Delta_P)$ intersects $U$ by Theorem~\ref{crossing}. Consequently, $f|_V$ has an analytic continuation $f$ to $\bH_g$ by Proposition~\ref{wolf}. The modularity of $f$ follows from the assumption that $f \in \FMA$, and Lemma~\ref{toy}.
\end{proof}

\begin{theorem}\label{ulysses}
For every integer $g \geq 2$, the graded algebra $\MG$ is algebraically closed in the graded algebra $\FMG$.
\end{theorem}
\begin{proof}
First we show that if a finite sum $f = \sum_{k} f_k \in \FMG$ for $f_k \in \FMA$ is algebraic over $\MG$, then so is each $f_k$. Indeed, as $f$ satisfies a polynomial equation over $\MG$, the highest weight piece $f_{k_0}$ of $f$ must satisfy the equation defined by the highest weight part, so $f_{k_0}$ is also algebraic over $\MG$, and we proceed by induction. We assume now that $f \in \FMA$ and $h f$ is integral over $\MG$ for some Hermitian modular form $h$.  By Proposition~\ref{institute}, $h f$ defines a Hermitian modular form.  In particular, $f$ defines a meromorphic function on $\bH_g$.

Furthermore, Proposition~\ref{wind} implies that $f$ is holomorphic on $\pi^{-1} (U) \subseteq \bH_g$ for some open neighborhood $U$ of $\partial Y_g$ under the natural map $\pi: \bH_g \lra X_g$. If the polar set of $f$ is nonempty, then the pushforward of it under $\pi$ does not intersect $U$, which is a contradiction by Theorem~\ref{crossing}. Hence $f$ is holomorphic on $\bH_g$, and defines a Hermitian modular form, as desired.
\end{proof}

\section{Modularity of symmetric formal Fourier--Jacobi series}
\label{jump}
In this section, we establish the main result, Theorem~\ref{heaven}, via a few steps of reduction based on the machinery introduced in Section~\ref{walk}. Let $E/\Q$ be a norm-Euclidean imaginary quadratic fields, over which the Hermitian modular forms are defined. 
\begin{theorem}\label{tess}
Let $g \geq 2$ be an integer, then
\begin{linenomath}
\begin{gather*}
    \FMG = \MG
    \tx{.}
\end{gather*}
\end{linenomath}
\end{theorem}
\begin{proof}
Combining Theorem~\ref{lolita} and Theorem~\ref{ulysses}, we complete the proof.
\end{proof}

We say that a point $\tau$ in the Hermitian upper half space $\bH_g$ is an elliptic fixed point if its stabilizer is strictly larger than the center $Z := Z \big(\UGZ \big)$ under the action of $\UGZ$ on $\bH_g$. We call $\ga \in \UGZ$ an elliptic element if $\ga \notin Z$ and has a fixed point in $\bH_g$.

\begin{lemma}\label{shelf}
Let $g \geq 2$ be an integer. Then, the codimension of elliptic fixed points on $\bH_g$ under the action of $\UGR$ is at least $g$. In particular, any holomorphic function on the complement of elliptic fixed points has an analytic continuation to the entire $\bH_g$.
\end{lemma}
\begin{proof}
We work with the open unit disk model, where the unitary group
\begin{linenomath}
 \begin{gather*}
     \rmU := \big\{\ga \in \GL{2g} (\C) :  \ga^\ast \begin{psmatrix}I_g & 0_g \\ 0_g & -I_g\end{psmatrix} \ga = \begin{psmatrix}I_g & 0_g \\ 0_g & -I_g\end{psmatrix}\big\}
 \end{gather*}
\end{linenomath}
acts on $\bD_g = \big\{\sigma \in \mathrm{Mat}_g (\C): \sigma^\ast \sigma < I_g \big\} (\cong \bH_g)$. First we observe that conjugate elements in $\rmU$ have the same fixed points up to $\rmU$-translations, which have the same codimension. For every elliptic element $\ga$, since $\ga$ is of finite order, it is $\rmU$-conjugate to some elements in the standard maximal compact subgroup $\rmU_g(\R) \times \rmU_g(\R)$. Therefore, it suffices to show the assertion for diagonal matrices $\ga = \mathrm{diag} (\lambda_1, \lambda_2, \ldots \lambda_{2g}) \in \rmU \setminus Z(\rmU)$. In fact, for any fixed diagonal matrix $\ga = \mathrm{diag} (\lambda_i)_i  \in \rmU \setminus Z(\rmU)$, a point $w = (w_{i, j})_{i, j} \in \bD_g$ is fixed by $\ga$ if and only if the equations $\lambda_i w_{i, j} = \lambda_{j + g} w_{i, j}$ hold for all $i, j \in \big\{1, 2, \ldots, g \big\}$. In particular, any elliptic point fixed by $\ga$ lies in the hyperplane cut out by the equations $w_{i,j} = 0$ for the pairs $(i, j)$ such that $\lambda_i \neq \lambda_{j + g}$. But the number of these pairs for such $\ga$ is always at least $g$, as desired. The rest is clear from Riemann's first extension theorem. 
\end{proof}

Recall that for any $f \in \MR$, the values that $f$ can take are in the $\UGZ$-invariant subspace 
\[V (\rho, k) := \Big\{ v \in V (\rho): \forall \ga \in Z\,
\big( j(\ga, \tau)^k \rho (\ga) v = v \big) \Big\} \tx{.}\]
In the next proposition, we relate the space of values of Hermitian modular forms of weight $k$ and type $\rho$ to $V (\rho, k)$.

\begin{proposition}\label{dog}
Let $g \in \Z_{\geq 2}$, and let $\rho$ be a finite-dimensional representation of $\UGZ$. Then, for any integer $k$, there exists some integer $k_0$ such that for every $\tau \in \bH_g$ that is not an elliptic fixed point, the values of $\mathrm{M}_{k_0}^{(g)} (\rho)$ at $\tau$ linearly span the space $V(\rho, k)$.
\end{proposition}
\begin{proof}
Let $X$ denote the Satake compactification of $Y_g = \UGZ \backslash \Hg$. By the work of Satake and Baily-Borel, there is some positive integer $k^\prime$, such that the $\cO_X$-module $\cF_{k^\prime}$ of weight-$k^\prime$ Hermitian modular forms of trivial type is a very ample invertible sheaf on $X$. Let $\cF_{k, \rho}$ denote the $\cO_X$-module of weight-$k$ Hermitian modular forms of type $\rho$, which is a coherent sheaf. By a result of Serre (\cite{hartshorne-1977}, Theorem 5.17), for large enough integers $n$, the sheaf $\cF_{k, \rho} \otimes_{\cO_X} \cF_{k^\prime}^{\otimes n} = \cF_{k + k^\prime n, \rho}$ is generated by a finite number of global sections. We set $k_0 = k + k^\prime n$  for some fixed $n$ so that $V (\rho, k) = V (\rho, k_0)$ and that $\cF_{k_0, \rho}$ is generated by finitely many $f_i \in \cF_{k_0, \rho} (X)$. Taking stalks at any point $\tau \in \bH_g$ that is not an elliptic fixed point, we deduce that the stalk $\cF_{k_0, \rho, \tau} = \cO_{X, \tau} \otimes_\C V (\rho, k)$ is generated by $\{f_{i, \tau}\}$. In particular, the values $f_i (\tau)$ linearly span $V (\rho, k)$.
\end{proof}

By restricting $\rho$ on the invariant subspace $V (\rho, k)$ in the rest of this section, we may assume $V(\rho) = V(\rho, k)$. Let $g \geq 2$ be an integer, and $\sigma, \sigma^\prime$ be two complex finite dimensional representations of the Hermitian modular group $\UGZ$. Recall a canonical pairing $\langle\, , \, \rangle$ induced by evaluation
\[
\langle\, , \, \rangle: \rmM_{k}^{(g)} (\sigma^\vee \otimes \sigma^\prime) 
\times \mathrm{FM}_{k^\prime}^{(g)} (\sigma)
\lra \mathrm{FM}_{k + k^\prime}^{(g)} (\sigma^\prime)
\tx{.}
\]
Moreover, the pairing $\langle\, , \, \rangle$ can be defined for meromorphic Hermitian modular forms and meromorphic formal symmetric Fourier--Jacobi series; see \cite{bruinier-2015}, Definition 2.3 for further details. For this aim, we need to assume that $\rho$ factors through a finite quotient.

\begin{proposition}\label{home} 
Let $g, l$ be integers such that $g \geq 2$ and $1 \leq l \leq g - 1$. If \, $\FMGL = \MG$, then
\begin{linenomath}
 \begin{gather*}
    \FMGR = \MGR 
 \end{gather*}
\end{linenomath}
for every complex finite dimensional representation $\rho$ of $\UGZ$ that factors through a finite quotient.
\end{proposition}

\begin{proof}
We follow the proof in \cite{bruinier-2015}, Theorem 1.2. Let $f \in \FMR$ be an arbitrary symmetric Fourier--Jacobi series of weight $k$, and let $n$ be the dimension of $V(\rho)$. First we fix an integer $k_0$ by Lemma~\ref{dog}, such that at every $\tau$ which is not an elliptic fixed point, the values of $ \rmM_{k_0}^{(g)}(\rho^\vee)$ generate the space $V(\rho^\vee)$, and we also fix such a point $\tau$.

Let $\rho_0^n$ be the trivial representation on $\C^n$. By assumption, there is some tuple of Hermitian modular forms $F = (F_i)_i \in \big({\rmM_{k_0}^{(g)}(\rho^\vee)}\big)^n$ such that the values $F_i (\tau) \in V (\rho^\vee)$ for $i = 1, \ldots, n$ are linearly independent. This defines a Hermitan modular form of weight $k_0$ and type $\rho^\vee \otimes \rho_0^n$, also denoted by $F$, as well as a meromorphic Hermitian modular form $F^{-1}$ of weight $- k_0$ valued in $\Hom_\C (\C^n, V(\rho))$, via the assignment $\tau^\prime \lmto \big(F_i(\tau^\prime) \big)_i^{-1}$ for the nondegenerate matrix $\big(F_i(\tau^\prime) \big)_i$.

We consider the pairing $\widetilde{f} := \big\langle F^{-1}, \langle F, f \rangle \big\rangle$.  By the assumption $\FMGL = \MG$, it is clear that $\widetilde{f}$ is a meromorphic Hermitian modular form of genus $g$, weight $k$, and type $\rho$, and its formal Fourier--Jacobi expansion matches $f$. Since $\widetilde{f}$ is holomorphic on an open neighborhood of $\tau$ by our choice of $F$, varying $\tau$ we find $f$ is holomorphic on the complement of the elliptic fixed points. By Lemma~\ref{shelf}, $f$ is holomorphic on $\bH_g$, and its modularity follows from that of $\widetilde{f}$.
\end{proof}

\begin{proposition}\label{holmes}
Let $g \geq 2$ be an integer. Assume $\FMGP = \MGP$ for every integer $g^\prime$ such that $2 \leq g^\prime \leq g - 1$ and for every complex finite dimensional representation $\rho$ of $\rmU(g^\prime, g^\prime) (\Z)$ that factors through a finite quotient. If $\FMG = \MG$, then $\FMGL = \MG$ for every integer $l$ satisfying $1 \leq l \leq g - 1$.
\end{proposition}

\begin{proof}
We prove the statement by induction on the cogenus $l$. The case of $l = 1$ is true for all $g \geq 2$ by assumption. For the case of $l \geq 2$, we fix $l^\prime := l - 1$. Assuming $\mathrm{FM}_\bullet^{(g^\prime, l^\prime)} = \mathrm{M}_\bullet^{(g^\prime)}$ for all $g^\prime \geq l$, we have to show that $\FMGL = \MG$ for all $g \geq l + 1$.

Let $f \in \FMC$ be an arbitrary symmetric Fourier--Jacobi series of cogenus $l$. To show $f \in \mathrm{FM}_\bullet^{(g, l^\prime)} = \MG$, we consider its formal Fourier--Jacobi expansion of cogenus $l^\prime$ defined in Subsection~\ref{Alice}. In fact, for $\tau = \begin{psmatrix}
\tau_1 & w \\
z & \tau_2
\end{psmatrix} \in \bH_g$ where $\tau_2 \in \bH_{l^\prime}$, the formal Fourier--Jacobi expansion
\begin{linenomath}
\begin{gather}\label{Seine}
    f (\tau) = \sum_{m^\prime \in \mathrm{Herm}_{l^\prime} (E)_{\geq 0}}
    \psi_{m^\prime} (\tau_1, w, z) e (m^\prime \tau_2)
\end{gather}
\end{linenomath}
holds as an identity of formal Fourier series, and is invariant under the weight-$k$ slash action of the Hermitian modular group $\UGZ$, as $f \in \FMC$. Therefore, it suffices to show the convergence of each $\psi_{m^\prime}$. 

If $m^\prime$ is nondegenerate, then $\psi_{m^\prime}$ admits a formal theta decomposition with coefficients $(h_{m^\prime, s^\prime})_{s^\prime}  \in \mathrm{FM}_{k - l^\prime}^{(g - l^\prime)} (\rho_{m^\prime}^{(g - l^\prime)})$, by Proposition~\ref{boutique}.  The assumption implies that $(h_{m^\prime, s^\prime})_{s^\prime}$ converges, and hence so does $\psi_{m^\prime}$. 

If $m^\prime$ is degenerate, we observe that for $0 \in \mathrm{Herm}_1 (E)_{\geq 0}$ and for every $m^{\prime\prime} \in \mathrm{Herm}_{l^\prime - 1} (E)_{\geq 0}$, $\psi_0$ and $\psi_{\begin{psmatrix}m^{\prime\prime} & 0 \\
0 & 0\end{psmatrix}}$ satisfy a relation in the form of Equation~\eqref{John}. By Proposition~\ref{finnegan} and the induction assumption, we have $\psi_0 \in \rmM_k^{(g - 1)}$, hence $\psi_{\begin{psmatrix}m^{\prime\prime} & 0 \\
0 & 0\end{psmatrix}}$ must be a Hermitian Jacobi form. Furthermore, for any degenerate matrix $m^\prime \in \mathrm{Herm}_{l^\prime} (E)_{\geq 0}$, there is a suitable matrix $u^\prime \in \GL{l^\prime} (\OK)$ such that ${u^\prime}^\ast m^\prime u^\prime$ is of the form $\begin{psmatrix}m^{\prime\prime} & 0 \\
0 & 0\end{psmatrix}$. By the invariance of \eqref{Seine} under the embedded action of  $\begin{psmatrix}
I_{g - l^\prime} & 0 \\
0 & u^\prime \end{psmatrix} \in \GL{g} (\OK)$, the convergence of $\psi_{m^\prime}$ is then reduced to that of $\psi_{\begin{psmatrix}m^{\prime\prime} & 0 \\
0 & 0\end{psmatrix}}$.
\end{proof}
\begin{remark}
As is shown in the proof, we may only assume in the proposition that $\FMGP = \MGP$ for every integer $g^\prime$ such that $2 \leq g^\prime \leq g - 1$ and for every Weil representation $\rho$ of $\rmU(g^\prime, g^\prime) (\Z)$.
\end{remark}

\begin{theorem}\label{heaven}
Let $E/\Q$ be a norm-Euclidean imaginary quadratic field, over which the Hermitian modular groups are defined. Let $g, l$ be integers such that $g \geq 2$ and $1 \leq l \leq g - 1$. Let $k$ be an integer and $\rho$ be a finite dimensional representation of $\UGZ$ that factors through a finite quotient. Then,
\begin{linenomath}
\begin{gather*}
    \FMR = \MR
    \tx{.}
\end{gather*}
\end{linenomath}
\end{theorem}
\begin{proof}
Combining Theorem~\ref{tess} and a successive application of Propositions~\ref{home}, \ref{holmes}, and \ref{home} again, we complete the proof.
\end{proof}

\section{Application: Kudla's modularity conjecture for unitary Shimura varieties}\label{sit}
In this section, we discuss more details for an application of Theorem~\ref{heaven} in the context of Kudla's conjecture for unitary Shimura varieties, following Section 1 in \cite{kudla-2004} and Section 3 in \cite{MR2928563}. We adopt the adelic language for the benefit of not handling the class number problem explicitly.

Let $E/\Q$ be an imaginary quadratic field and fix a complex embedding $E \lhra \C$. Let $n \geq 2$ be an integer. Let $V$ be an $E$-Hermitian space of signature $(n - 1, 1)$, and $H = \mathrm{U} (V)$ be the unitary group of $V$, which is a reductive group over $\Q$. Let $\bD$ be the Hermitian symmetric space attached to $H (\bR)$, which in a projective model consists of all negative $\C$-lines in $V (\C)$. For any compact open subgroup $K \subseteq H (\bA_f)$, these data determine a Shimura variety $X_K = \mathrm{Sh}_K (H, \bD)$, which is a smooth quasi-projective variety of dimension $n - 1$ defined over the reflex field $E$, whose complex points $X_K (\C)$ can be identified with the locally symmetric space
\begin{linenomath}
 \begin{gather*}
     H (\Q)\backslash \big(\bD \times H (\bA_f) / K\big)\tx{.}
 \end{gather*}
\end{linenomath}

We then proceed to define special cycles. Let $g$ be an integer such that $1 \leq g \leq n - 1$. For a $g$-tuple $\lambda = (\lambda_i)_i \in V(E)^g$, we define $V_\lambda := \{v \in V: \forall i (v \perp \lambda_i) \}$, $H_\lambda := \mathrm{U} (V_\lambda)$, and $\bD_\lambda := \{w \in \bD: \forall i (w \perp \lambda_i) \}$. For every compact open subgroup $K \subseteq H (\bA_f)$ and every coset $\ga \in K \backslash H (\bA_f)$, these data determine a ``translated'' algebraic cycle $Z_K(\lambda, \ga)$ on $X_K$ over $E$ at level $K$ via the map
\begin{linenomath}
 \begin{align*}
     H_\lambda (\Q) \backslash \big( \bD_\lambda \times H_\lambda (\bA_f) 
     / H_\lambda (\bA_f) \cap \ga^{-1} K \ga \big) & \lra 
     H (\Q) \backslash \big( \bD \times H (\bA_f) / K \big)  = X_K (\C)\tx{,} \\
     [(w, \ga^\prime)] & \lmto [(w, \ga^\prime \ga^{-1})]
     \tx{.}
 \end{align*}
\end{linenomath}

Note that $Z_K(\lambda, \ga)$ vanishes unless the $g \times g$ Hermitian matrix $Q (\lambda) := \frac{1}{2} \big( (\lambda_i, \lambda_j) \big)_{i, j}$ is positive semidefinite and $\mathrm{rk} (Q (\lambda)) = \mathrm{rk} \{\lambda_i: 1 \leq i \leq g\}$. Given a nice weight function, special cycles are a family of weighted linear combinations of these translated cycles over essentially integral tuples $\lambda$, indexed by the values of $Q (\lambda)$. More precisely, this family is defined as follows. Let $T \in \mathrm{Herm}_{g} (E)_{\geq 0}$ be a Hermitian positive semidefinite matrix, and $\varphi$ be a $K$-invariant Schwartz-Bruhat function on the finite adelic points $V (\bA_{E,f})^g$ over $E$.  If $V (E)$ is nonempty, for a fixed rational tuple $\lambda_0 \in V (E)^g$,  and the shift $\ga \in K \backslash H (\bA_f)$ such that $\ga \lambda_0 = \lambda$, we define the translated cycle $Z_K (\lambda)$ to be $Z_K(\lambda_0, \ga)$. The special cycle $Z_K (T, \varphi)$ in $X_K$ is then defined to be the finite sum
\begin{linenomath}
 \begin{gather*}
     Z_K (T, \varphi) := \sum_{\lambda \in K \backslash \Omega_{T} 
     (\bA_{E, f})} \varphi (\lambda) Z_K (\lambda)
     \tx{,}
\end{gather*}
\end{linenomath}
for the set $\Omega_{T}  (\bA_{E, f}) := \{\lambda \in V (\bA_{E, f})^g : Q (\lambda) = T\}$. Note that the finiteness of the sum is due to the compactness of the set $\Omega_{T}  (\bA_{E, f}) \cap \mathrm{supp} (\varphi)$. Moreover, the special cycles $Z_K (T, \varphi)$ as $K$ varies are compatible in the projective system under the \'etale covering maps of the corresponding Shimura varieties $X_K$, so we may write $Z (T, \varphi)$ for $Z_K (T, \varphi)$. 

We are in position to state the unitary Kudla modularity conjecture. Let $\cL_{\bD}$ be the tautological line bundle over $\bD$, that is, the restriction to $\bD$ of the line bundle $\cO (- 1)$ over $\bP (V (\C))$. Since the action of $H (\bR)$ on $\bD$ lifts naturally to  a certain action on $\cL_{\bD}$, the line bundle $ \cL_{\bD}$ descends to a line bundle $\cL_K$ over the Shimura variety $X_K$. Just like the special cycles $Z_K (T, \varphi)$, the line bundles $\cL_K$ are also compatible in the projective system of Shimura varities $X_K$ as $K$ varies, and we write $\cL^\vee$ for the class of the dual bundle in the first Chow group $\mathrm{CH}^1 (X_K)$. It is clear that $Z (T, \varphi)$ is a cycle of codimension $r (T)$ (the rank of $T$), and we write $Z (T, \varphi)$ also for its class in the Chow group $\mathrm{CH}^{r (T)} (X_K)$. Finally, the generating series for special cycles on $X_K$ is defined to be the formal sum
\begin{linenomath}
 \begin{gather*}
 \psi_{g, \varphi}^{\mathrm{CH}} (\tau) :=
\sum_{T \in \mathrm{Herm}_g (E)_{\geq 0}}     
Z (T, \varphi) \cdot (\cL^\vee )^{g - r (T)} q^T
\tx{,}
 \end{gather*}
\end{linenomath}
where the dot ``$\cdot$'' denotes the product in the Chow ring $\mathrm{CH}^{\bullet} (X_K)$, and $q^T = \exp \big( 2 \pi i\, \mathrm{tr} (T \tau) \big)$ for $\tau$ in the Hermitian upper half space $\Hg$.
The unitary Kudla's modularity conjecture states that $\psi_{g, \varphi}^{\mathrm{CH}} (\tau)$ is a Hermitian modular form over $E/\Q$ of genus $g$ and weight $n$ valued in the Chow group $\mathrm{CH}^{g} (X_K)$. 

An immediate consequence of this conjecture is that the $\C$-span of special cycles $Z (T, \varphi)$ in the complexification $\mathrm{CH}^{g} (X_K)_\C$ is finite dimensional, and a further appealing feature is that relations between Fourier coefficients of certain Hermitian modular forms give rise to the corresponding relations between special cycles $Z (T, \varphi)$ in $\mathrm{CH}^{g} (X_K)_\C$. See \cite{bruinier-raum-2015, raum-2015c} for this computational aspect in the orthogonal case. Under the assumption that the generating series $\psi_{g, \varphi}^{\mathrm{CH}} (\tau)$ is absolutely convergent, Y. Liu proved the conjecture in Theorem 3.5 of \cite{MR2928563}.  Combining his result with Theorem~\ref{heaven}, we obtain the following result.

\begin{theorem}\label{Paradise}
In the cases of norm-Euclidean imaginary quadratic fields, the unitary Kudla conjecture is true for open Shimura varieties.
\end{theorem}
\begin{proof}
We claim that the generating series $\psi_{g, \varphi}^{\mathrm{CH}} (\tau)$ is a symmetric formal Fourier--Jacobi series of genus $g$, cogenus $g - 1$, and weight $n$, and the statement then follows from Theorem~\ref{heaven} for $l = g - 1$, $k = n$, and the trivial type $\rho$. First, writing $\tau = \begin{psmatrix}\tau_1 & w \\ z & \tau_2 \end{psmatrix}$, we consider the formal Fourier--Jacobi expansion
\begin{linenomath}
 \begin{gather*}
     \psi_{g, \varphi}^{\mathrm{CH}} (\tau) 
     =
     \sum_{m \in \mathrm{Herm}_{g - 1} (E)_{\geq 0}}
     \phi_{m, \varphi} (\tau_1, w, z) e (m \tau_2)
     \tx{,}
 \end{gather*}
\end{linenomath}
where the $m$-th formal Fourier--Jacobi coefficient $\phi_{m, \varphi}$ is given by
\begin{linenomath}
 \begin{gather*}
     \phi_{m, \varphi}(\tau_1, w, z)
     =
     \sum_{\substack{n \in \Q_{\geq 0}\\r \in \mathrm{Mat}_{1, g - 1} (E)}}
     Z \Bigg(\begin{pmatrix}n & r \\ r^\ast & m \end{pmatrix}, \varphi \Bigg)
     e (n \tau_1 + r z)
     e (r^\ast w)
     \tx{.}
 \end{gather*}
\end{linenomath}
To verify our claim, by the first part of Theorem 3.5 in \cite{MR2928563}, it suffices to show the absolute convergence of $\phi_{m, \varphi}$, and we proceed as in W. Zhang's thesis \cite{zhang-2009}. Although connected components of a Shimura variety $X_K$ are in general defined over cyclotomic extensions of $\Q$, they are linked via Galois action, and special cycles on them are Galois conjugate to each other. Therefore, without loss of generality, we may work with a connected Shimura varieties in the classical setting as in \cite{borcherds-1999}, and unless otherwise stated, we use from now on the numbering and notation as in Chapter 2 of \cite{zhang-2009} for the counterparts that can be similarly defined in the unitary case. The Shimura variety $X_K$ we work with can be written as $X_\Gamma$ for a neat congruence subgroup $\Gamma$ of $\rmU (L)$, where $L$ is the unitary counterpart of the lattice defined in Theorem 2.5.

First, we note that our formal Fourier--Jacobi coefficients $\phi_{m, \varphi}$ correspond to $\theta_{(\lambda, \mu), t}$ defined in (2.7), where the functions $F$ are defined as $F_\iota$ just before Theorem 2.9 for linear functionals $\iota$ on the Chow group $\mathrm{CH}^r_L$ and $1\leq r = g \leq n - 1$. By (2.10), which expresses $\theta_{(\lambda, \mu), t}$ as a finite sum of $\theta_{\lambda,\underline{x}}$ (defined just before (2.10)), and following the proof of Proposition 2.6, we see that each formal Fourier--Jacobi coefficient $\theta_{(\lambda, \mu), t}$ is absolutely convergent assuming the absolute convergence of the vector-valued generating $q$-series $\Theta_{F_{\underline{x}}} = \Theta_{F_{\iota, \underline{x}}}$ (defined just before Theorem 2.5) on $\bH_1$ for every tuple $\underline{x} \in L^{\vee, r - 1}$ with positive definite $\Q \underline{x}$, and $1\leq r \leq n - 1$. Following the proof of Theorem 2.9, for any linear functional $\iota$ on the Chow group $\mathrm{CH}^r_L$, in the unitary case we also have the equation $\Theta_{F_{\iota, \underline{x}}} =
\Theta_{F_{\iota \circ i \underline{x}}}$, where $i \underline{x}$ is the associated shifting by a power of dual line bundle defined in the middle of the proof. Since $\iota \circ i\underline{x}$ is a linear functional on the first Chow group $\mathrm{CH}^1_{L_{\underline{x}}}$, the absolute convergence of the counterpart of $\Theta_{F_{\iota \circ i\underline{x}}}$ in the unitary case follows from the second part of Theorem 3.5 in \cite{MR2928563}, which completes the proof of our claim.
\end{proof}

\begin{remark}
If the Hermitian space $V$ is anisotropic, then $X_K$ is compact and our theorem is already complete in this case. If $V$ is isotropic so that $X_K$ is not compact, we only obtain the result for open Shimura varieties. On a toroidal compactification $X_K$, modularity of the generating series of special cycles becomes a natural question only after we define the boundary components of special cycles. The problem was posed in \cite{kudla-2004}, and in the case of special divisors on toroidal compactifications of orthogonal Shimura varieties, Jan--Zemel proved the corresponding modularity conjecture in their recent preprint \cite{Bruinier--Zemel}.
\end{remark}

\section*{Acknowledgements}
The author wishes to express his gratitude to his main advisor, Martin Raum, for introducing this fascinating topic and many useful discussions. He would also like to thank Yifeng Liu for pointing out the notion of norm--Euclidean number fields which recovers two cases carelessly missed in a first draft. The author also thanks Yota Maeda for informing his work \cite{maeda-2020a, Maeda-2021} on the generalized Kudla modularity conjectures.

%%%%%%%%%%%%%%%%%%%%%%%%%%%%%%%%%%%%%%%%%%%%%%%%%%
%%% BIBLIOGRAPHY

\renewbibmacro{in:}{}
\renewcommand{\bibfont}{\normalfont\small\raggedright}
\renewcommand{\baselinestretch}{.8}

\Needspace*{4em}
\begin{multicols}{2}
\printbibliography[heading=none]%[heading=bibnumbered]

@book {Wang-2020,
    AUTHOR = {Wang, Yuxiang},
     TITLE = {Theta {F}unctions and {H}ermitian {J}acobi {F}orms over
              {I}maginary {Q}uadratic {F}ields},
      NOTE = {Thesis (Ph.D.)--Northwestern University},
 PUBLISHER = {ProQuest LLC, Ann Arbor, MI},
      YEAR = {2020},
     PAGES = {43},
      ISBN = {979-8664-79831-9},
   MRCLASS = {Thesis},
  MRNUMBER = {4172115},
       URL =
              {http://gateway.proquest.com/openurl?url_ver=Z39.88-2004&rft_val_fmt=info:ofi/fmt:kev:mtx:dissertation&res_dat=xri:pqm&rft_dat=xri:pqdiss:28030931},
}

@article{maeda-2020a, title={The Modularity of Special Cycles on Orthogonal Shimura Varieties over Totally Real Fields under the Beilinson–Bloch Conjecture}, DOI={10.4153/S000843952000020X}, journal={Canadian Mathematical Bulletin}, publisher={Canadian Mathematical Society}, author={Maeda, Yota}, year={2020}, pages={1–15}}

@misc{Maeda-2021,
    author = {Maeda, Yota},
     title = {Modularity of special cycles on unitary Shimura varieties over CM-fields},
 howpublished = {arXiv:2101.09232v3},
      year = 2021,}

@misc{Kudla-2019,
    author = {Kudla, Stephen S.},
     title = {Remarks on generating series for special cycles},
 howpublished = { arXiv: 1908.08390v1},
      year = 2019, }

@misc{Bruinier--Zemel,
    author = {Bruinier, Jan Hendrik and Zemel, Shaul},
     title = {Special Cycles on Toroidal Compactifications of Orthogonal Shimura Varieties},
 howpublished = { arXiv:1912.11825v2},
      year = 2020, }

@misc{Bruinier--Howard--Kudla--Rapoport--Yang-1,
    author = {Bruinier, Jan and Howard, Benjamin and  S.Kudla, Kudla and Rapoport, Michael and  Yang, Tonghai},
     title = {Modularity of generating series of divisors on unitary Shimura varieties},
 howpublished = {arXiv:1702.07812v3},
      year = 2020, }

@misc{Bruinier--Howard--Kudla--Rapoport--Yang-2,
    author = {Bruinier, Jan and Howard, Benjamin and  S.Kudla, Kudla and Rapoport, Michael and  Yang, Tonghai},
     title = {Modularity of generating series of divisors on unitary Shimura varieties II: arithmetic applications},
 howpublished = {arXiv:1710.00628v2},
      year = 2020, }

@article {deligne-1996,
    AUTHOR = {Deligne, P.},
     TITLE = {Extensions centrales de groupes alg\'{e}briques simplement
              connexes et cohomologie galoisienne},
   JOURNAL = {Inst. Hautes \'{E}tudes Sci. Publ. Math.},
  FJOURNAL = {Institut des Hautes \'{E}tudes Scientifiques. Publications
              Math\'{e}matiques},
    NUMBER = {84},
      YEAR = {1996},
     PAGES = {35--89 (1997)},
      ISSN = {0073-8301},
   MRCLASS = {14L15 (14F20 19C09)},
  MRNUMBER = {1441006},
MRREVIEWER = {Andy R. Magid},
       URL = {http://www.numdam.org/item?id=PMIHES_1996__84__35_0},
}

@article {brylinski-deligne-2001,
    AUTHOR = {Brylinski, Jean-Luc and Deligne, Pierre},
     TITLE = {Central extensions of reductive groups by {$\bold K_2$}},
   JOURNAL = {Publ. Math. Inst. Hautes \'{E}tudes Sci.},
  FJOURNAL = {Publications Math\'{e}matiques. Institut de Hautes \'{E}tudes
              Scientifiques},
    NUMBER = {94},
      YEAR = {2001},
     PAGES = {5--85},
      ISSN = {0073-8301},
   MRCLASS = {20G15 (19C09)},
  MRNUMBER = {1896177},
MRREVIEWER = {Emmanuel Peyre},
       DOI = {10.1007/s10240-001-8192-2},
       URL = {https://doi.org/10.1007/s10240-001-8192-2},
}

@article {prasad-2004,
    AUTHOR = {Prasad, Gopal},
     TITLE = {Deligne's topological central extension is universal},
   JOURNAL = {Adv. Math.},
  FJOURNAL = {Advances in Mathematics},
    VOLUME = {181},
      YEAR = {2004},
    NUMBER = {1},
     PAGES = {160--164},
      ISSN = {0001-8708},
   MRCLASS = {20G25 (20G10)},
  MRNUMBER = {2020658},
MRREVIEWER = {Herbert Abels},
       DOI = {10.1016/S0001-8708(03)00048-3},
       URL = {https://doi.org/10.1016/S0001-8708(03)00048-3},
}

@article {prasad-rapinchuk,
    AUTHOR = {Prasad, Gopal and Rapinchuk, Andrei S.},
     TITLE = {Computation of the metaplectic kernel},
   JOURNAL = {Inst. Hautes \'{E}tudes Sci. Publ. Math.},
  FJOURNAL = {Institut des Hautes \'{E}tudes Scientifiques. Publications
              Math\'{e}matiques},
    NUMBER = {84},
      YEAR = {1996},
     PAGES = {91--187 (1997)},
      ISSN = {0073-8301},
   MRCLASS = {22E55 (20G35)},
  MRNUMBER = {1441007},
MRREVIEWER = {Lawrence Morris},
       URL = {http://www.numdam.org/item?id=PMIHES_1996__84__91_0},
}

@incollection {Hulek-Sakaran-2002,
    AUTHOR = {Hulek, Klaus and Sankaran, G. K.},
     TITLE = {The geometry of {S}iegel modular varieties},
 BOOKTITLE = {Higher dimensional birational geometry ({K}yoto, 1997)},
    SERIES = {Adv. Stud. Pure Math.},
    VOLUME = {35},
     PAGES = {89--156},
 PUBLISHER = {Math. Soc. Japan, Tokyo},
      YEAR = {2002},
   MRCLASS = {11F46 (14G35)},
  MRNUMBER = {1929793},
MRREVIEWER = {Rolf Berndt},
       DOI = {10.2969/aspm/03510089},
       URL = {https://doi.org/10.2969/aspm/03510089},
}

@article {shimura-1978a,
    AUTHOR = {Shimura, Goro},
     TITLE = {On certain reciprocity-laws for theta functions and modular
              forms},
   JOURNAL = {Acta Math.},
  FJOURNAL = {Acta Mathematica},
    VOLUME = {141},
      YEAR = {1978},
    NUMBER = {1-2},
     PAGES = {35--71},
      ISSN = {0001-5962},
   MRCLASS = {10D20 (12A65)},
  MRNUMBER = {491518},
MRREVIEWER = {Larry J. Goldstein},
       DOI = {10.1007/BF02545742},
       URL = {https://doi.org/10.1007/BF02545742},
}

@article {jamazaki-1986,
    AUTHOR = {Yamazaki, Tadashi},
     TITLE = {Jacobi forms and a {M}aass relation for {E}isenstein series},
   JOURNAL = {J. Fac. Sci. Univ. Tokyo Sect. IA Math.},
  FJOURNAL = {Journal of the Faculty of Science. University of Tokyo.
              Section IA. Mathematics},
    VOLUME = {33},
      YEAR = {1986},
    NUMBER = {2},
     PAGES = {295--310},
      ISSN = {0040-8980},
   MRCLASS = {11F46 (11F55)},
  MRNUMBER = {866395},
MRREVIEWER = {G\"{u}nter K\"{o}hler},
}

@article {murase-1989,
    AUTHOR = {Murase, Atsushi},
     TITLE = {{$L$}-functions attached to {J}acobi forms of degree {$n$}.
              {I}. {T}he basic identity},
   JOURNAL = {J. Reine Angew. Math.},
  FJOURNAL = {Journal f\"{u}r die Reine und Angewandte Mathematik. [Crelle's
              Journal]},
    VOLUME = {401},
      YEAR = {1989},
     PAGES = {122--156},
      ISSN = {0075-4102},
   MRCLASS = {11F55 (11F46 11S40 22E55)},
  MRNUMBER = {1018057},
MRREVIEWER = {Rolf Berndt},
       DOI = {10.1515/crll.1989.401.122},
       URL = {https://doi.org/10.1515/crll.1989.401.122},
}

@incollection {zagier-1981,
    AUTHOR = {Zagier, D.},
     TITLE = {Sur la conjecture de {S}aito-{K}urokawa (d'apr\`es {H}.
              {M}aass)},
 BOOKTITLE = {Seminar on {N}umber {T}heory, {P}aris 1979--80},
    SERIES = {Progr. Math.},
    VOLUME = {12},
     PAGES = {371--394},
 PUBLISHER = {Birkh\"{a}user, Boston, Mass.},
      YEAR = {1981},
   MRCLASS = {10D24 (10D20)},
  MRNUMBER = {633910},
MRREVIEWER = {A. N. Andrianov},
}

@article {MR2114161,
    AUTHOR = {Knopp, Marvin and Mason, Geoffrey},
     TITLE = {Vector-valued modular forms and {P}oincar\'{e} series},
   JOURNAL = {Illinois J. Math.},
  FJOURNAL = {Illinois Journal of Mathematics},
    VOLUME = {48},
      YEAR = {2004},
    NUMBER = {4},
     PAGES = {1345--1366},
      ISSN = {0019-2082},
   MRCLASS = {11F11 (30F35)},
  MRNUMBER = {2114161},
MRREVIEWER = {Thomas R. Shemanske},
       URL = {http://projecteuclid.org/euclid.ijm/1258138515},
}

@article {MR871665,
    AUTHOR = {Eie, Min King},
     TITLE = {A dimension formula for {H}ermitian modular cusp forms of
              degree two},
   JOURNAL = {Trans. Amer. Math. Soc.},
  FJOURNAL = {Transactions of the American Mathematical Society},
    VOLUME = {300},
      YEAR = {1987},
    NUMBER = {1},
     PAGES = {61--72},
      ISSN = {0002-9947},
   MRCLASS = {11F55 (11F72)},
  MRNUMBER = {871665},
MRREVIEWER = {G\"{u}nter K\"{o}hler},
       DOI = {10.2307/2000588},
       URL = {https://doi.org/10.2307/2000588},
}

@article {MR387496,
    AUTHOR = {Borel, Armand},
     TITLE = {Stable real cohomology of arithmetic groups},
   JOURNAL = {Ann. Sci. \'{E}cole Norm. Sup. (4)},
  FJOURNAL = {Annales Scientifiques de l'\'{E}cole Normale Sup\'{e}rieure. Quatri\`eme
              S\'{e}rie},
    VOLUME = {7},
      YEAR = {1974},
     PAGES = {235--272 (1975)},
      ISSN = {0012-9593},
   MRCLASS = {22E40 (20G10)},
  MRNUMBER = {387496},
MRREVIEWER = {H. Garland},
       URL = {http://www.numdam.org/item?id=ASENS_1974_4_7_2_235_0},
}

@article {MR3406827,
    AUTHOR = {Bruinier, Jan Hendrik and Westerholt-Raum, Martin},
     TITLE = {Kudla's modularity conjecture and formal {F}ourier-{J}acobi
              series},
   JOURNAL = {Forum Math. Pi},
  FJOURNAL = {Forum of Mathematics. Pi},
    VOLUME = {3},
      YEAR = {2015},
     PAGES = {e7, 30},
   MRCLASS = {11F46 (14C25)},
  MRNUMBER = {3406827},
MRREVIEWER = {Haigang Zhou},
       DOI = {10.1017/fmp.2015.6},
       URL = {https://doi.org/10.1017/fmp.2015.6},
}

@article {MR2928563,
    AUTHOR = {Liu, Yifeng},
     TITLE = {Arithmetic theta lifting and {$L$}-derivatives for unitary
              groups, {I}},
   JOURNAL = {Algebra Number Theory},
  FJOURNAL = {Algebra \& Number Theory},
    VOLUME = {5},
      YEAR = {2011},
    NUMBER = {7},
     PAGES = {849--921},
      ISSN = {1937-0652},
   MRCLASS = {11G18 (11F27 11G50 20G05)},
  MRNUMBER = {2928563},
MRREVIEWER = {Stefano Vigni},
       DOI = {10.2140/ant.2011.5.849},
       URL = {https://doi.org/10.2140/ant.2011.5.849},
}

@incollection {Haverkamp-thesis-1995,
    AUTHOR = {Haverkamp, Klaus},
     TITLE = {Hermitesche {J}acobiformen},
 BOOKTITLE = {Schriftenreihe des {M}athematischen {I}nstituts der
              {U}niversit\"{a}t {M}\"{u}nster. 3. {S}erie, {V}ol. 15},
    SERIES = {Schriftenreihe Math. Inst. Univ. M\"{u}nster 3. Ser.},
    VOLUME = {15},
     PAGES = {105},
 PUBLISHER = {Univ. M\"{u}nster, Math. Inst., M\"{u}nster},
      YEAR = {1995},
   MRCLASS = {11F55 (11F72)},
  MRNUMBER = {1335247},
MRREVIEWER = {Aloys Krieg},
}

@article {MR1377681,
    AUTHOR = {Haverkamp, Klaus K.},
     TITLE = {Hermitian {J}acobi forms},
   JOURNAL = {Results Math.},
  FJOURNAL = {Results in Mathematics. Resultate der Mathematik},
    VOLUME = {29},
      YEAR = {1996},
    NUMBER = {1-2},
     PAGES = {78--89},
      ISSN = {0378-6218},
   MRCLASS = {11F55 (11F27 11F46 32N10)},
  MRNUMBER = {1377681},
MRREVIEWER = {Rolf Berndt},
       DOI = {10.1007/BF03322207},
       URL = {https://doi.org/10.1007/BF03322207},
}

@book{hartshorne-1977,
    AUTHOR = {Hartshorne, Robin},
     TITLE = {Algebraic geometry},
      NOTE = {Graduate Texts in Mathematics, No. 52},
 PUBLISHER = {Springer-Verlag, New York-Heidelberg},
      YEAR = {1977},
     PAGES = {xvi+496}, }

@book{ash-mumford-rapoport-tai-1975,
    AUTHOR = {Ash, A. and Mumford, D. and Rapoport, M. and Tai, Y.},
     TITLE = {Smooth compactification of locally symmetric varieties},
      NOTE = {Lie Groups: History, Frontiers and Applications, Vol. IV},
 PUBLISHER = {Math. Sci. Press, Brookline, Mass.},
      YEAR = {1975},
     PAGES = {iv+335}, }

@article{bruinier-raum-2015,
    AUTHOR = {Bruinier, Jan Hendrik and Westerholt-Raum, Martin},
     TITLE = {Kudla's modularity conjecture and formal {F}ourier-{J}acobi
              series},
   JOURNAL = {Forum Math. Pi},
  FJOURNAL = {Forum of Mathematics. Pi},
    VOLUME = {3},
      YEAR = {2015},
     PAGES = {e7, 30}, }

@article{andrianov-1979,
   author    = {Andrianov, Anatoli N.},
   title     = {Modular descent and the {S}aito-{K}urokawa conjecture},
    journal  = {Invent. Math.},
    volume   = {53},
    number   = {3},
    year     = 1979,
    pages    = {267--280}
}

@article{borcherds-1998,
  AUTHOR =	 {Borcherds, Richard E.},
  TITLE =	 "Automorphic forms with singularities on Grassmannians",
  JOURNAL =	 {Invent. Math.},
  FJOURNAL =	 {Inventiones Mathematicae},
  VOLUME =	 132,
  YEAR =	 1998,
  NUMBER =	 3,
  PAGES =	 {491--562},
}

@article{borcherds-1999,
    AUTHOR = {Borcherds, Richard E.},
     TITLE = {The {G}ross-{K}ohnen-{Z}agier theorem in higher dimensions},
   JOURNAL = {Duke Math. J.},
  FJOURNAL = {Duke Mathematical Journal},
    VOLUME = {97},
      YEAR = {1999},
    NUMBER = {2},
     PAGES = {219--233}}

@book{borel-wallach-1980,
    AUTHOR = {Borel, Armand and Wallach, Nolan R.},
     TITLE = {Continuous cohomology, discrete subgroups, and representations
              of reductive groups},
    SERIES = {Annals of Mathematics Studies},
    VOLUME = {94},
 PUBLISHER = {Princeton University Press and University of Tokyo Press},
   ADDRESS = {Princeton NJ and Tokyo},
      YEAR = {1980},
     PAGES = {xvii+388},
}

@article{braun-1949,
  author={Braun, Hel},
  title={{H}ermitian modular functions},
  journal="Ann. of Math. ",
  volume=50,
  number=2,
  pages={827--855},
  year=1949,
}

@article{braun-1950,
    AUTHOR = {Braun, Hel},
     TITLE = {Hermitian modular functions. {II}.},
   JOURNAL = {Ann. of Math. (2)},
  FJOURNAL = {Annals of Mathematics. Second Series},
    VOLUME = {51},
      YEAR = {1950},
     PAGES = {92--104},
}

@article{braun-1951,
    AUTHOR = {Braun, Hel},
     TITLE = {Hermitian modular functions {III}},
   JOURNAL = {Ann. of Math. (2)},
  FJOURNAL = {Annals of Mathematics. Second Series},
    VOLUME = {53},
      YEAR = {1951},
     PAGES = {143--160},
}

@article{bruinier-2015,
    AUTHOR = {Bruinier, Jan Hendrik},
     TITLE = {Vector valued formal {F}ourier-{J}acobi series},
   JOURNAL = {Proc. Amer. Math. Soc.},
  FJOURNAL = {Proceedings of the American Mathematical Society},
    VOLUME = {143},
      YEAR = {2015},
    NUMBER = {2},
     PAGES = {505--512},
}

@book{bruinier-van-der-geer-harder-zagier-2008,
    AUTHOR = {Bruinier, Jan Hendrik and van der Geer, Gerard and Harder,
              G{\"u}nter and Zagier, Don B.},
     TITLE = {The 1-2-3 of modular forms},
    SERIES = {Universitext},
      NOTE = {Lectures from the Summer School on Modular Forms and their
              Applications held in Nordfjordeid, June 2004},
    editor = {Ranestad, Kristian},
 PUBLISHER = {Springer-Verlag, Berlin},
      YEAR = {2008},
     PAGES = {x+266}}

@book{eichler-zagier-1985,
    AUTHOR = {Eichler, Martin and Zagier, Don B.},
     TITLE = {The theory of {J}acobi forms},
    SERIES = {Progress in Mathematics},
    VOLUME = {55},
 PUBLISHER = {Birkh\"auser Boston Inc.},
   ADDRESS = {Boston, MA},
      YEAR = {1985},
     PAGES = {v+148},
}

@article{gross-kohnen-zagier-1987,
   author   = {Gross, Benedict and Kohnen, Winfried and Zagier, Don B.},
   title    = {Heegner points and derivatives of {$L$}-series. {II}},
   journal  = {Math. Ann.},
   FJOURNAL = {Mathematische Annalen},
   volume   = 278,
   number   = {1-4},
   year     = 1987,
   pages    = {497--562}
}

@article {kudla-1981,
    AUTHOR = {Kudla, Stephen S.},
     TITLE = {On certain {E}uler products for {${\mathrm SU}(2,\,1)$}},
   JOURNAL = {Compositio Math.},
  FJOURNAL = {Compositio Mathematica},
    VOLUME = {42},
      YEAR = {1980/81},
    NUMBER = {3},
     PAGES = {321--344},
      ISSN = {0010-437X},
   MRCLASS = {10D20},
  MRNUMBER = {607374},
MRREVIEWER = {Hiroshi Saito},
       URL = {http://www.numdam.org/item?id=CM_1980__42_3_321_0},
}

@article{hirzebruch-zagier-1976,
    AUTHOR = {Hirzebruch, F. and Zagier, Don B.},
     TITLE = {Intersection numbers of curves on {H}ilbert modular surfaces
              and modular forms of {N}ebentypus},
   JOURNAL = {Invent. Math.},
  FJOURNAL = {Inventiones Mathematicae},
    VOLUME = {36},
      YEAR = {1976},
     PAGES = {57--113},
}

@article{ibukiyama-poor-yuen-2012,
    AUTHOR = {Ibukiyama, Tomoyoshi and Poor, Cris and Yuen, David S.},
     TITLE = {Jacobi forms that characterize paramodular forms},
   JOURNAL = {Abh. Math. Semin. Univ. Hambg.},
  FJOURNAL = {Abhandlungen aus dem Mathematischen Seminar der Universit\"at
              Hamburg},
    VOLUME = {83},
      YEAR = {2013},
    NUMBER = {1},
     PAGES = {111--128}}

@book{klingen-1990,
  AUTHOR =	 {Klingen, Helmut},
  TITLE =	 {Introductory lectures on {S}iegel modular forms},
  SERIES =	 {Cambridge Studies in Advanced Mathematics},
  VOLUME =	 {20},
  PUBLISHER =	 {Cambridge University Press},
  ADDRESS =	 {Cambridge},
  YEAR =	 {1990},
  PAGES =	 {x+162},
}

@article{kudla-1997,
    AUTHOR = {Kudla, Stephen S.},
     TITLE = {Algebraic cycles on {S}himura varieties of orthogonal type},
   JOURNAL = {Duke Math. J.},
  FJOURNAL = {Duke Mathematical Journal},
    VOLUME = {86},
      YEAR = {1997},
    NUMBER = {1},
     PAGES = {39--78}}

@incollection{kudla-2004,
    AUTHOR = {Kudla, Stephen S.},
     TITLE = {Special cycles and derivatives of {E}isenstein series},
 BOOKTITLE = {Heegner points and {R}ankin {$L$}-series},
    SERIES = {Math. Sci. Res. Inst. Publ.},
    VOLUME = {49},
     PAGES = {243--270},
 PUBLISHER = {Cambridge Univ. Press, Cambridge},
      YEAR = {2004},
   MRCLASS = {11G18 (11F37 11F46 11F67 14G40)},
  MRNUMBER = {2083214 (2005g:11108)},
MRREVIEWER = {Alexey A. Panchishkin},
       DOI = {10.1017/CBO9780511756375.009},
       URL = {http://dx.doi.org/10.1017/CBO9780511756375.009},
}

@article{kudla-millson-1986,
    AUTHOR = {Kudla, Stephen S. and Millson, John},
     TITLE = {The theta correspondence and harmonic forms. {I}},
   JOURNAL = {Math. Ann.},
  FJOURNAL = {Mathematische Annalen},
    VOLUME = {274},
      YEAR = {1986},
    NUMBER = {3},
     PAGES = {353--378},
}

@article{kudla-millson-1987,
    AUTHOR = {Kudla, Stephen S. and Millson, John},
     TITLE = {The theta correspondence and harmonic forms. {II}},
   JOURNAL = {Math. Ann.},
  FJOURNAL = {Mathematische Annalen},
    VOLUME = {277},
      YEAR = {1987},
    NUMBER = {2},
     PAGES = {267--314},
}

@article{kudla-millson-1990,
    AUTHOR = {Kudla, Stephen S. and Millson, John},
     TITLE = {Intersection numbers of cycles on locally symmetric spaces and
              {F}ourier coefficients of holomorphic modular forms in several
              complex variables},
   JOURNAL = {Inst. Hautes \'Etudes Sci. Publ. Math.},
  FJOURNAL = {Institut des Hautes \'Etudes Scientifiques. Publications
              Math\'ematiques},
    NUMBER = {71},
      YEAR = {1990},
     PAGES = {121--172}}

@article{maass-1979a,
  AUTHOR =	 {Maass, Hans},
  TITLE =	 {\"{U}ber eine {S}pezialschar von {M}odulformen
                  zweiten {G}rades},
  JOURNAL =	 {Invent. Math.},
  FJOURNAL =	 {Inventiones Mathematicae},
  VOLUME =	 {52},
  YEAR =	 {1979},
  NUMBER =	 {1},
  PAGES =	 {95--104},
}

@article{maass-1979b,
  AUTHOR =	 {Maass, Hans},
  TITLE =	 {\"{U}ber eine {S}pezialschar von {M}odulformen
                  zweiten {G}rades. {II}},
  JOURNAL =	 {Invent. Math.},
  FJOURNAL =	 {Inventiones Mathematicae},
  VOLUME =	 {53},
  YEAR =	 {1979},
  NUMBER =	 {3},
  PAGES =	 {249--253},
}

@article{maass-1979c,
  AUTHOR =	 {Maass, Hans},
  TITLE =	 {\"{U}ber eine {S}pezialschar von {M}odulformen
                  zweiten {G}rades. {III}},
  JOURNAL =	 {Invent. Math.},
  FJOURNAL =	 {Inventiones Mathematicae},
  VOLUME =	 {53},
  YEAR =	 {1979},
  NUMBER =	 {3},
  PAGES =	 {255--265},
}

@book{matsumura-1980,
    AUTHOR = {Matsumura, Hideyuki},
     TITLE = {Commutative algebra},
    SERIES = {Mathematics Lecture Note Series},
    VOLUME = {56},
   EDITION = {Second},
 PUBLISHER = {Benjamin/Cummings Publishing Co.},
   ADDRESS = {Reading, Mass.},
      YEAR = {1980},
     PAGES = {xv+313},
}

@article {murase-sugano-2007,
    AUTHOR = {Murase, Atsushi and Sugano, Takashi},
     TITLE = {On the {F}ourier-{J}acobi expansion of the unitary {K}udla
              lift},
   JOURNAL = {Compos. Math.},
  FJOURNAL = {Compositio Mathematica},
    VOLUME = {143},
      YEAR = {2007},
    NUMBER = {1},
     PAGES = {1--46},
      ISSN = {0010-437X},
   MRCLASS = {11F55 (11F67 11F70)},
  MRNUMBER = {2295193},
MRREVIEWER = {Min Ho Lee},
       DOI = {10.1112/S0010437X06002491},
       URL = {https://doi.org/10.1112/S0010437X06002491},
}

@article {murase-sugano-2002,
    AUTHOR = {Murase, Atsushi and Sugano, Takashi},
     TITLE = {Fourier-{J}acobi expansion of {E}isenstein series on unitary
              groups of degree three},
   JOURNAL = {J. Math. Sci. Univ. Tokyo},
  FJOURNAL = {The University of Tokyo. Journal of Mathematical Sciences},
    VOLUME = {9},
      YEAR = {2002},
    NUMBER = {2},
     PAGES = {347--404},
      ISSN = {1340-5705},
   MRCLASS = {11F55 (11F67)},
  MRNUMBER = {1904935},
MRREVIEWER = {Tonghai Yang},
}

@book{piatetsky-shapiro-1966,
    AUTHOR = {Piatetsky-Shapiro, I. I.},
     TITLE = {G\'eom\'etrie des domaines classiques et th\'eorie des
              fonctions automorphes},
    SERIES = {Traduit du Russe par A. W. Golovanoff. Travaux et Recherches
              Math\'ematiques, No. 12},
 PUBLISHER = {Dunod},
   ADDRESS = {Paris},
      YEAR = {1966},
     PAGES = {iv+160}}

@article{raum-2015c,
    AUTHOR = {Westerholt-Raum, Martin},
     TITLE = {Formal {F}ourier {J}acobi expansions and special cycles of
              codimension two},
   JOURNAL = {Compos. Math.},
  FJOURNAL = {Compositio Mathematica},
    VOLUME = {151},
      YEAR = {2015},
    NUMBER = {12},
     PAGES = {2187--2211}, }

@PhdThesis{zhang-2009,
  author = 	 {Zhang, Wei},
  title = 	 {{Modularity of Generating Functions of Special Cycles on Shimura Varieties}},
  school = 	 {Columbia University},
  year = 	 2009}

@article{ziegler-1989,
    AUTHOR = {Ziegler, Claus Dieter},
     TITLE = {Jacobi forms of higher degree},
   JOURNAL = {Abh. Math. Sem. Univ. Hamburg},
  FJOURNAL = {Abhandlungen aus dem Mathematischen Seminar der Universit\"at
              Hamburg},
    VOLUME = 59,
      YEAR = 1989,
     PAGES = {191--224}}
\end{multicols}

% In case publishers don't support biblatex, switch to amsref (or plain bibtex) 
% \bibliographystyle{alpha}
% \bibliography{bibliography.bib}

%% vim: spell spelllang=en_us
\end{document}